\RequirePackage{fix-cm}
\documentclass[smallextended]{svjour3}       
\smartqed  
\usepackage{graphicx}
\usepackage{cite}
\usepackage{amsmath}
\usepackage{amsfonts}
\usepackage{amssymb}
\journalname{Information Geometry}

\newtheorem{condition}[theorem]{Condition}

\begin{document}

\title{Logarithmic divergences from optimal transport and R\'{e}nyi geometry
}

\titlerunning{Logarithmic divergences and R\'{e}nyi geometry}        

\author{Ting-Kam Leonard Wong
}


\institute{Ting-Kam Leonard Wong \at
              University of Toronto\\
              \email{tkl.wong@utoronto.ca}           
}

\date{Received: date / Accepted: date}

\maketitle

\begin{abstract}
Divergences, also known as contrast functions, are distance-like quantities defined on manifolds of non-negative or probability measures. Using the duality in optimal transport, we introduce and study the one-parameter family of $L^{(\pm \alpha)}$-divergences. It includes the Bregman divergence corresponding to the Euclidean quadratic cost, and the $L$-divergence introduced by Pal and the author in connection with portfolio theory and a logarithmic cost function. They admit natural generalizations of exponential family that are closely related to the $\alpha$-family and $q$-exponential family. In particular, the $L^{(\pm \alpha)}$-divergences of the corresponding potential functions are R\'{e}nyi divergences. Using this unified framework we prove that the induced geometries are dually projectively flat with constant sectional curvatures, and a generalized Pythagorean theorem holds true. Conversely, we show that if a statistical manifold is dually projectively flat with constant curvature $\pm \alpha$ with $\alpha > 0$, then it is locally induced by an $L^{(\mp \alpha)}$-divergence. We define in this context a canonical divergence which extends the one for dually flat manifolds.

\end{abstract}

\section{Introduction} \label{sec:intro}
In this paper we study geometric properties of statistical manifolds defined by solutions to some optimal transport problems. We show that they characterize dually projectively flat manifolds with constant curvatures and provide natural geometries to generalizations of exponential family. Before stating our main results let us describe the background and motivations of our study.

\subsection{Background}
Information geometry studies manifolds of probability distributions and measures using geometric ideas. For modern introductions to its theory and applications we refer the reader to the recent monographs \cite{A16} and \cite{AJLS17}. Typically the geometry is induced by a divergence which is a distance-like quantity $\mathbf{D}\left[q : p \right] \geq 0$ defined on the underlying manifold $\mathcal{M}$. Of particular importance is the Bregman divergence. Given a differentiable concave function $\varphi(\xi)$ defined on a convex domain $\Omega$ in $\mathbb{R}^d$, the Bregman divergence of $\varphi$ is defined for $\xi, \xi' \in \Omega$ by
\begin{equation} \label{eqn:Bregman.divergence}
\mathbf{D}^{(0+)}\left[ \xi : \xi'\right] = D \varphi(\xi') \cdot ( \xi - \xi' ) - \left( \varphi(\xi) - \varphi(\xi')\right),
\end{equation}
where $D$ is the Euclidean gradient and $a \cdot b$ is the dot product. (The meaning of the superscript $0+$ will become clear in Section \ref{sec:summary}.) For example, if $\mathcal{M}$ is an exponential family of probability densities where $\xi$ is the natural parameter, the relative entropy, also known as the Kullback-Leibler divergence,
can be expressed as the Bregman divergence of the cumulant generating function $\varphi(\xi)$ (which is convex, so in \eqref{eqn:Bregman.divergence} we may consider $-\varphi$ or use $\mathbf{D}^{(0-)}$ to be defined in Section \ref{sec:log.divergence}). The differential geometry of Bregman divergence was first studied by Nagaoka and Amari in \cite{NA82}.

Following the general framework established by Eguchi \cite{E83, E92}, the geometry induced on $\mathcal{M}$ by a given divergence $\mathbf{D}\left[ \cdot : \cdot\right]$ consists of a Riemannian metric $g$ and a dual pair $(\nabla, \nabla^*)$ of torsion-free affine connections (see Definition \ref{def:dualistic.structure}). We call the quadruplet $(\mathcal{M}, g, \nabla, \nabla^*)$ a dualistic structure or a statistical manifold. The Riemannian metric captures the quadratic approximation of $\mathbf{D}\left[ q : p \right]$ when $q \approx p$, whereas the two affine connections capture higher order local properties and allow us to define primal and dual geodesics that are in some sense compatible with the divergence. In contrast with usual Riemannian geometry where the Levi-Civita connection is the canonical connection defined by the metric, in information geometry we usually use a pair of connections $(\nabla, \nabla^*)$ to quantify the  asymmetry of the divergence: generally we have $\mathbf{D}\left[p : q\right] \neq \mathbf{D}\left[q : p \right]$. When the divergence is symmetric (for example when $\mathbf{D}\left[q : p\right] = \frac{1}{2} d(p, q)^2$ where $d$ is a metric) both $\nabla$ and $\nabla^*$ coincide with the Levi-Civita connection. From this geometric perspective, the Bregman divergence is fundamental in the sense that it is the canonical divergence which generates a dually flat geometry, i.e., both the primal and dual connections $\nabla$ and $\nabla^*$ have zero curvature (see for example \cite[Section 6.6]{A16} and \cite[Section 4.2]{AJLS17}; this is also a limiting case of Theorem \ref{thm:main3}). Moreover, the Bregman divergence satisfies a generalized Pythagorean theorem which enables explicit computation and geometric interpretation of projections based on the divergence. By now these results are well-known and numerous applications can be found in the two monographs cited above.

Motivated by portfolio theory, in a series of papers \cite{PW13, PW14, W15, PW16, W17, PW18} we introduced and studied the $L$-divergence ($L$ stands for logarithmic) defined by
\begin{equation} \label{eqn:L.1.divergence}
\mathbf{D}^{(1)}\left[ \xi : \xi'\right] = \log (1 +  D \varphi(\xi') \cdot ( \xi - \xi' )) - \left( \varphi(\xi) - \varphi(\xi')\right),
\end{equation}
where now the potential function $\varphi$ is exponentially concave, i.e, $\Phi = e^{\varphi}$ is concave. Being the logarithm of the concave function $\Phi$, $\varphi$ is itself concave and the $L$-divergence has a logarithmic correction term that captures the extra concavity. An important example of $L$-divergence is the ``excess growth rate" defined for probability vectors $p, q$ by
\begin{equation} \label{eqn:excess.growth.rate}
\mathbf{D}\left[ q : p \right] = \log \left( \sum_{i = 0}^d \frac{1}{d + 1} \frac{q_i}{p_i} \right) - \sum_{i = 0}^d \frac{1}{d + 1} \log \frac{q_i}{p_i},
\end{equation}
where $\varphi(p) = \sum_{i = 0}^{d} \frac{1}{d + 1} \log p_i$. Exponentially concave functions appear naturally in many recent applications in analysis, probability and statistics. In particular, let us mention the paper \cite{EKS15} on Bochner's inequality on metric measure spaces. Other examples can be found in the references of \cite{PW16}.

In \cite{PW16} we established the following results where the underlying manifold is the open unit simplex:
\begin{itemize}
\item[(i)] The dualistic structure induced by the $L$-divergence \eqref{eqn:L.1.divergence} is dually projectively flat with constant curvature $-1$, and the generalized Pythagorean theorem (which has a financial interpretation in terms of optimal trading frequency) holds true.
\item[(ii)] The $L$-divergence can be expressed in terms of the optimal transport map with respect to a logarithmic cost function, whereas the Bregman divergence corresponds to the classical quadratic cost. Furthermore, displacement interpolations for these transport problems are related to the dual geodesics of the induced dualistic structure.
\end{itemize}
It is natural to ask if there is a unified framework that covers both Bregman and $L$-divergences. Moreover, can we characterize statistical manifolds that satisfy these (very strong) geometric properties? On the other hand, Amari and Nagaoka \cite{AN02} considered $\alpha$-families of probability distributions and the corresponding $\alpha$-divergences that generalize the exponential family and the Kullback-Leibler divergence. We will see that these questions are intimately related.

The connection between optimal transport and information geometry is a recent topic that has started to receive serious attention. For example, the recent paper \cite{AKO17} studies the manifold of entropy-relaxed optimal transport plans and defines a divergence which interpolates between the Kullback-Leibler divergence and the Wasserstein metric (optimal transport cost). Dynamic approach are considered in \cite{CPSV15} and the recent papers \cite{LM18, CL18}. Also see \cite{P17} which builds on the ideas of \cite{PW16} and embeds a large family of optimal transport problems in a generalized simplex.

\subsection{Summary of main results} \label{sec:summary}
The key objects of study in this paper are the $L^{(\pm \alpha)}$-divergences defined for $\alpha > 0$. The $L^{(\alpha)}$-divergence is defined by
\begin{equation} \label{eqn:L.alpha.divergence}
\mathbf{D}^{(\alpha)} \left[ \xi : \xi'\right] = \frac{1}{\alpha} \log (1 +  \alpha \nabla \varphi(\xi') \cdot ( \xi - \xi' )) - \left( \varphi(\xi) - \varphi(\xi')\right),
\end{equation}
where the function $\varphi$ is $\alpha$-exponentially concave, i.e., $e^{\alpha \varphi}$ is concave. To cover the case of positive curvature, we will also consider a local $L^{(-\alpha)}$-divergence $\mathbf{D}^{(-\alpha)}$ where $\varphi$ is $\alpha$-exponentially convex, i.e., $e^{\alpha \varphi}$ is convex. When $\alpha = 1$ we have the $L^{(1)}$-divergence \eqref{eqn:L.1.divergence}, and when $\alpha \rightarrow 0^+$ it reduces to the Bregman divergence \eqref{eqn:Bregman.divergence} of a concave function ($\mathbf{D}^{(0-)}$ is the Bregman divergence of a convex function). Based on financial arguments, the $L^{(\alpha)}$-divergence was introduced recently in \cite{W17} as the canonical interpolation between the Bregman and $L$-divergences. In this paper we consider also the $L^{(-\alpha)}$-divergence and show that they are of fundamental importance. We note that there exists other natural extrapolations of the Bregman divergence.\footnote{The author thanks an anonymous referee for pointing this out.} For example, Zhang \cite{Z04} (also see \cite{Z05, Z15, NZ18}) introduced a parameterized family $\mathcal{D}^{(\alpha)}_{\Phi}$ of divergences for a given convex function $\Phi$, which includes the Bregman and $\alpha$-divergences among others, and introduced the concepts of referential duality and representational duality. The biduality of divergence also plays a role in this paper (see Corollary \ref{cor:c.divergence.duality}).

While we worked on the open unit simplex in our previous paper \cite{PW16}, here we consider a general open convex domain $\Omega$ in $\mathbb{R}^d$. This answers a question asked in \cite[Section 1]{PW16} about extending the results to general domains. By a translation if necessary, we may assume that $0 \in \overline{\Omega}$ (this is used for the normalization of the function $c^{(\alpha)}$ in \eqref{eqn:cost.intro}). Nevertheless, we note that exponential concavity/convexity impose certain restrictions on the domains of $\varphi$ and the $L^{(-\alpha)}$-divergence.

In Sections \ref{sec:c.divergence} and \ref{sec:log.divergence} we prove the following result which connects the $L^{(\pm \alpha)}$-divergence with optimal transport and generalizes the classical Legendre duality for Bregman divergence.

\begin{theorem} \label{thm:main1}
For $\alpha > 0$, consider the function
\begin{equation} \label{eqn:cost.intro}
c^{(\alpha)}(x, y) = \frac{1}{\alpha} \log \left(1 + \alpha x \cdot y \right),
\end{equation}
where $x \cdot y$ is the Euclidean dot product. Then the $L^{(\alpha)}$-divergence \eqref{eqn:L.alpha.divergence} admits the self-dual representation
\begin{equation} \label{eqn:canonical.divergence.intro}
\mathbf{D}^{(\alpha)} \left[ \xi : \xi'\right] = c^{(\alpha)}(\xi, \eta') - \varphi(\xi) - \psi(\eta'),
\end{equation}
where $\psi$ is the $\alpha$-conjugate of $\varphi$ and is (locally) $\alpha$-exponentially concave, and $\eta = D^{(\alpha)} \varphi(\xi)$ is the $\alpha$-gradient which generalizes the Legendre transformation.

Similarly, when $\varphi$ is $\alpha$-exponentially convex, the $L^{(-\alpha)}$-divergence has the self-dual representation
\begin{equation} \label{eqn:canonical.divergence.intro2}
\mathbf{D}^{(-\alpha)} \left[ \xi : \xi'\right] = \varphi(\xi) + \psi(\eta') -  c^{(\alpha)}(\xi, \eta'),
\end{equation}
where now $\varphi$ and $\psi$ are $\alpha$-exponentially convex.
\end{theorem}

Precise statements corresponding to Theorem \ref{thm:main1} are given in Theorems \ref{thm:L.divergence.as.c.divergence}, \ref{prop:alpha.duality}, \ref{thm:alpha.duality} and \ref{thm:duality.minus.alpha}. The self-dual expressions \eqref{eqn:canonical.divergence.intro} and \eqref{eqn:canonical.divergence.intro2} are motivated by optimal transport which identifies the dual coordinate system $\eta$ as the image of $\xi$ under the optimal transport map. This idea comes from our previous paper \cite{PW16}. 

Motivated by the relationships between Bregman divergence and exponential family (see \cite{BMDG05}), in Section \ref{sec:q.exponential} we consider generalizations of exponential family that are closely related to the $\alpha$-family and $q$-exponential family \cite{AN02, N10, AO11}. We show that the analogue of the cumulant generating function is $\alpha$-exponentially concave/convex (Propositions \ref{prop:check.exp.concave} and \ref{prop:check.exp.concave2}), and the corresponding $L^{(\pm \alpha)}$-divergence is the R\'{e}nyi divergence (Theorem \ref{thm:renyi}). Moreover, the dual function is the R\'{e}nyi entropy. Our results thus provide a new approach to the geometry of R\'{e}nyi  and $\alpha$-divergences. When $\alpha \rightarrow 0$ we recover the dually flat geometry of exponential family. The case of finite simplex is discussed in Section \ref{sec:alpha.divergence}.

In Sections \ref{sec:L.geometry} and \ref{sec:geometric.consequences} we study the statistical manifold $(\mathcal{M} = \Omega, g, \nabla, \nabla^*)$ induced by the $L^{(\pm \alpha)}$-diverence. Our approach, based on the duality in Theorem \ref{thm:main1}, gives a unified treatment covering both the Bregman and $L$-divergences and simplifies the proofs in \cite{PW16}. The following result summarizes Theorems \ref{prop:projective.flat}, \ref{thm:constant.curvature} and \ref{thm:pyth}.

\begin{theorem} \label{thm:main2}
For $\alpha > 0$, the dualistic structure induced by the $L^{(\pm \alpha)}$-divergence defined on an open convex set $\Omega \subset \mathbb{R}^d$ with $d \geq 2$ is dually projectively flat, has constant sectional curvature $\mp \alpha$, and the generalized Pythagorean theorem holds. (When $\alpha \rightarrow 0^+$ we reduce to the dually flat case.)
\end{theorem}

By dual projective flatness, we mean that there exist `affine' coordinate systems ($\xi$ and $\eta$) under which the primal and dual geodesics are straight lines up to time reparameterizations (for the precise technical statement see Definition \ref{def:projective.flatness}). Since the properties in Theorem \ref{thm:main2} are very strong and spaces of constant curvatures play fundamental roles in differential geometry, it is natural to ask if Theorem \ref{thm:main2} has a converse. Indeed there is one and it will be proved in Section \ref{sec:characterize.geometry} (see Theorems \ref{thm:projectively.flat.characterize} and \ref{thm:canonical.divergence}).

\begin{theorem}  \label{thm:main3}
Consider a dualistic structure $(\mathcal{M}, g, \nabla, \nabla^*)$ which is dually projectively flat and has constant curvature $-\alpha < 0$. Then locally there exist affine coordinate systems $\xi$ and $\eta$ for $\nabla$ and $\nabla^*$ respectively, and $\alpha$-exponentially concave functions $\varphi(\xi)$ and $\psi(\eta)$ which satisfy the generalized Fenchel identity
\[
\varphi(\xi) + \psi(\eta) \equiv c^{(\alpha)}(\xi, \eta).
\]
Moreover, the self-dual representation \eqref{eqn:canonical.divergence.intro} defines locally a canonical divergence which induces the given dualistic structure. Analogous statements hold when the curvature is $\alpha > 0$.
\end{theorem}

Our results establish the $L^{(\pm \alpha)}$-divergence as the canonical divergence for a dually projectively flat statistical manifold with constant curvature, thus generalizing the Bregman divergence for dually flat manifolds (see \cite{K90} for another characterization based on affine differential geometry). Recently, Ay and Amari \cite{AA15} (also see \cite{FA18}) defined a canonical divergence for an arbitrary statistical manifold and showed that on a dually flat manifold (i.e., $\alpha = 0$) it reduces to the Bregman divergence. It is interesting to know if their canonical divergence is consistent with our $L^{(\pm \alpha)}$-divergences. Since the $L^{(\pm \alpha)}$-divergences have properties  analogous to those of the Bregman divergence, explicit computations are tractable and a natural question is to study how they perform in applied problems such as clustering \cite{BMDG05} and statistical estimation. Projectively flat connections, statistical manifolds with constant curvatures and relations with affine differential geometry have been studied in the literature; see in particular \cite{K90, DNV90, K94, M99, CU14}. Using ideas of optimal transport and the $L^{(\pm \alpha)}$-divergence, we consider these properties as a whole and are able to give a new and elegant characterization of the geometry.

In optimal transport, it is well known that the quadratic cost, which corresponds in our setting to the Bregman divergence, is intimately related to Brownian motion and the heat equation (via the Wasserstein gradient flow of entropy, see \cite{JKO98, AGS08}) as well as large deviations \cite{ADPZ11} and the  Schr{\"o}dinger problem \cite{L12}. A programme suggested in \cite{PW16} is to study the analogues for our logarithmic cost functions, and in \cite{PW18} we took a first step by studying a multiplicative Schr{\"o}dinger problem. We plan to address some of these questions in future research. Another intriguing problem is to connect other information-geometric structures (such as curvature and the $\rho$-$\tau$ embedding, see \cite{Z04}) with optimal transport.

\section{Divergences induced by optimal transport maps} \label{sec:c.divergence}
A major theme of this paper is that the $L^{(\pm \alpha)}$-divergences are intimated related to optimal transport maps. As a motivating example we first consider the classical Bregman divergence which corresponds to the quadratic cost. Then we introduce the general framework of $c$-divergence following the ideas of \cite{PW16}. 

\subsection{The Monge-Kantorovich optimal transport problem}
We begin by introducing some basic terminologies of optimal transport. For further details we refer the reader to standard references such as \cite{AG13, S15, V03, V08}. Let $\mathcal{X}$ and $\mathcal{Y}$ be Polish spaces (i.e., complete and separable metric spaces) interpreted respectively as the source and target spaces of the optimal transport problem. Let $c: \mathcal{X} \times \mathcal{Y} \rightarrow \mathbb{R}$ be a continuous cost function. Given Borel probability measures $\mu \in \mathcal{P}(\mathcal{X})$ and $\nu \in \mathcal{P}(\mathcal{Y})$, the Monge-Kantorovich optimal transport problem is
\begin{equation} \label{eqn:monge.kantorovich}
\inf_{\gamma} \int_{\mathcal{X} \times \mathcal{Y}} c(x, y) d\gamma(x, y),
\end{equation}
where the infimum is taken over joint distributions $\gamma \in \mathcal{P}(\mathcal{X} \times \mathcal{Y})$ whose marginals are $\mu$ and $\nu$ respectively. We call $\gamma$ a coupling of the pair $(\mu, \nu)$ and write $\gamma \in \Pi(\mu, \nu)$. If $\gamma$ attains the infimum in \eqref{eqn:monge.kantorovich}, we say that it is an optimal coupling. For certain cost functions and under suitable conditions on $\mu$ and $\nu$, the optimal coupling has the form
\[
\gamma = (\mathrm{Id} \times T)_{\#} \mu
\]
for some measurable map $T : \mathcal{X} \rightarrow \mathcal{Y}$. Here we use $F_{\#} \mu = \mu \circ F^{-1}$ to denote the pushfoward of a measure. In this case the optimal coupling is deterministic, and we call $T$ an optimal transport map.

\begin{remark} \label{rem:cost.function.modify}
Note that if we replace the cost function $c$ by
\begin{equation} \label{eqn:c.extra.linear.terms}
\tilde{c}(x, y) = c(x, y) + h(x) + k(y),
\end{equation}
where $h$ and $k$ and real-valued functions on $\mathcal{X}$ and $\mathcal{Y}$ respectively, then for any coupling $\gamma \in \Pi(\mu, \nu)$ we have
\[
\int \tilde{c} d\gamma = \int c d\gamma + \int h d\mu + \int k d\nu.
\]
The last two integrals, if they exist, are determined once $\mu$ and $\nu$ are fixed and are independent of the coupling $\gamma$. This means that we are free to modify the cost function in the manner of \eqref{eqn:c.extra.linear.terms} without changing the optimal couplings.
\end{remark}

\subsection{Quadratic cost and Bregman divergence}
A fundamental example is where $\mathcal{X} = \mathcal{Y} = \mathbb{R}^d$ ($d \geq 1$) and
\begin{equation} \label{eqn:quadratic.cost}
c(x, y) = \frac{1}{2} \|x - y \|^2 = \frac{1}{2} \sum_{i = 1}^d (x^i - y^i)^2
\end{equation}
is the quadratic cost (where we write $x = (x^1, \ldots, x^d)$). Suppose $\mu, \nu \in \mathcal{P}(\mathbb{R}^d)$ have finite second moments, i.e.,
\[
\int \|x\|^2 d\mu(x) < \infty, \quad \int \|y\|^2 d\nu(y) < \infty,
\]
and suppose $\mu$ is absolutely continuous with respect to the Lebesgue measure on $\mathbb{R}^d$. By Brenier's theorem \cite{B91}, there exists a convex function $\varphi : \mathbb{R}^d \rightarrow \mathbb{R} \cup \{+\infty\}$ such that its gradient $D \varphi$ (which is defined $\mu$-a.e.) pushforwards $\mu$ to $\nu$. Moreover, the deterministic coupling
\[
\gamma = (\mathrm{Id} \times \nabla \varphi)_{\#} \mu
\]
solves the Monge-Kantorovich problem. In fact, it can be shown that the optimal transport map $T(x) = D\varphi(x)$, which is a Legendre trasnformation, is unique $\mu$-a.e. 

To see how the Bregman divergence comes into play, suppose that $\varphi$ is $C^2$ and $D^2 \varphi$ (the Hessian) is strictly positive definite. Then the Legendre transform $y = T(x)$ is a diffeomorphism (see \cite[Chapter 1]{A16}), and we have $x = D \varphi^*(y)$ and $\varphi^*$ is the convex conjugate of $\varphi$. By Fenchel's inequality, for any $x$ and $y'$ we have
\begin{equation} \label{eqn:Fenchel.inequality}
\varphi(x) + \varphi^*(y') \geq x \cdot y',
\end{equation}
and equality holds if and only if $y' = D \varphi(x)$. Intuitively, the inequality \eqref{eqn:Fenchel.inequality} quantifies the inefficiency of coupling $x' = D \varphi^*(y')$ with $y = D \varphi(x)$ when compared to $y'$. Since $\varphi^*(y') = x' \cdot y' - \varphi(x')$, we may rearrange \eqref{eqn:Fenchel.inequality} to get
\[
\varphi(x) - \varphi(x') - D\varphi(x') \cdot (x - x') \geq 0,
\]
which is nothing but the Bregman divergence of $\varphi$. 

\subsection{$c$-divergence}
Now we show that this idea can be formulated in an abstract framework using generalized concepts of convex analysis. In the following definition we fix a cost function $c: \mathcal{X} \times \mathcal{Y} \rightarrow \mathbb{R}$.

 \begin{definition}
Let $f: \mathcal{X} \rightarrow \mathbb{R} \cup \{-\infty\}$ and $g: \mathcal{Y} \rightarrow \mathbb{R} \cup \{-\infty\}$.
\begin{enumerate}
\item[(i)] The $c$-transforms of $f$ and $g$ are defined respectively by
\begin{equation} \label{eqn:c.transform}
\begin{split}
f^c(y) &= \inf_{x \in \mathcal{X}} \left(c(x, y) - f(x)\right), \quad y \in \mathcal{Y}, \\
g^c(x) &= \inf_{y \in \mathcal{Y}} \left(c(x, y) - g(x)\right), \quad x \in \mathcal{X}.
\end{split}
\end{equation}
\item[(ii)] We say that $f$ (respectively $g$) is $c$-concave if $f^{cc} = f$ (respectively $g^{cc} = g$).
\item[(iii)] If $f$ and $g$ are $c$-concave, their $c$-superdifferentials are defined by
\begin{equation} \label{eqn:c.super.differential}
\begin{split}
\partial^c f &= \{(x, y) \in \mathcal{X} \times \mathcal{Y}: f(x) + f^c(y) = c(x, y)\},\\
\partial^c g &= \{(x, y) \in \mathcal{X} \times \mathcal{Y}: g^c(x) + g(y) = c(x, y)\}.
\end{split}
\end{equation}
\end{enumerate}
\end{definition}

If $f$ is $c$-concave on $\mathcal{X}$, then
\begin{equation} \label{eqn:fenchel}
f(x) + f^c(y) \leq c(x, y)
\end{equation}
for all $(x, y) \in \mathcal{X} \times \mathcal{Y}$, and equality holds if and only if $(x, y) \in \partial^c f$. We call \eqref{eqn:fenchel} the (generalized) Fenchel inequality (or identity when equality holds) which generalizes \eqref{eqn:Fenchel.inequality}. The following result (see for example \cite[Theorem 2.13]{AG13}) is sometimes called the Fundamental Theorem of Optimal Transport.

\begin{theorem} \label{thm:ft.of.optimal.transport}
Suppose the cost function $c$ is continuous and bounded below, and there exist $a \in L^1(\mu)$ and $b \in L^1(\nu)$ such that $c(x, y) \leq a(x) + b(y)$. Then, a coupling $\gamma$ of the pair $(\mu, \nu)$ solves the Monge-Kantorovich problem \eqref{eqn:monge.kantorovich} if and only if there exists a $c$-concave function $f$ on $\mathcal{X}$ such that $\mathrm{supp}(\gamma) \subset \partial^c f$. 
\end{theorem}

By Theorem \ref{thm:ft.of.optimal.transport}, a $c$-concave function $f$ on $\mathcal{X}$ (or, rather, its $c$-superdifferential) can be regarded as encoding the solution to an optimal transport problem.

\medskip

Let $f$ is a $c$-concave function on $\mathcal{X}$. Assume that $f$ is $c$-differentiable in the following sense: for each $x \in \mathcal{X}$ there is a unique element $y \in \mathcal{Y}$ such that $(x, y) \in \partial^c f$. We call $y = D^c f(x)$ the $c$-gradient of $f$. This is a rather strong assumption but is satisfied for the cost functions considered in this paper. By Theorem \ref{thm:ft.of.optimal.transport} we may regard $D^c f: \mathcal{X} \rightarrow \mathcal{Y}$ as an optimal transport map. For a detailed discussion of the smoothness properties of optimal transport maps see \cite[Chapter 11]{V08}.

Using the definition of $c$-concavity, for any $x, x' \in \mathcal{X}$ we have
\begin{equation} \label{eqn:c.inequality}
f(x') + c(x, y') - c(x', y') \geq f(x), \quad y' = D^c f(x').
\end{equation}
Analogous definitions and results hold for a $c$-concave function $g$ on $\mathcal{Y}$. The inequality \eqref{eqn:c.inequality} motivates the definition of $c$-divergence (a more suggestive name may be transport divergence) which was first introduced in \cite{PW16}.

\begin{definition}[$c$-divergence]
Let $f$ be $c$-concave and $c$-differentiable on $\mathcal{X}$. The $c$-divergence of $f$ is the functional $\mathbf{D}_f: \mathcal{X} \times \mathcal{X} \rightarrow [0, \infty)$ defined by
\begin{equation} \label{eqn:c.divergence}
\mathbf{D}_f\left[x : x'\right] = c(x, y') - c(x', y') - (f(x) - f(x')), \quad x, x' \in \mathcal{X}, \quad y' = D^c f(x').
\end{equation}
If $g$ is a $c$-differentiable $c$-concave function on $\mathcal{Y}$, we define the dual $c$-divergence $\mathbf{D}_g: \mathcal{Y} \times \mathcal{Y} \rightarrow [0, \infty)$ by
\begin{equation} \label{eqn:dual.c.divergence}\mathbf{D}_g^*\left[y : y'\right] = c(x', y) - c(x', y') - (g(y) - g(y')), \quad y, y' \in \mathcal{Y}, \quad x' = D^c g(y').
\end{equation}
\end{definition}

\begin{theorem} [self-dual representations] \label{thm:c.divergence.self.dual}
Let $y' = D^c f(x')$. Then
\begin{equation} \label{eqn:self.dual.primal}
\mathbf{D}_f \left[x : x' \right] = c(x, y') - f(x) - f^c(y').
\end{equation}
Similarly, if $x = D^c g(y)$, then
\begin{equation} \label{eqn:self.dual.dual}
\mathbf{D}_g^* \left[y : y' \right] = c(x', y) - g(y) - g^c(x').
\end{equation}
\end{theorem}
\begin{proof}
By the generalized Fenchel identity \eqref{eqn:fenchel}, we have
\[
f(x') + f^c(y') = c(x', y') \Rightarrow f^c(y') = c(x', y') - f(x').
\]
It follows that
\begin{equation*}
\begin{split}
\mathbf{D}_f \left[x : x' \right] &= c(x, y') - c(x', y') - (f(x) - f(x')) \\
  &= c(x, y') - f(x) - f^c(y').
\end{split}
\end{equation*}
This gives \eqref{eqn:self.dual.primal}. The proof of the second identity \eqref{eqn:self.dual.dual} is similar.
\end{proof}


\begin{corollary} [Duality of $c$-divergence] \label{cor:c.divergence.duality}
Suppose $f^c$ is $c$-differentiable on the range of $D^c f$. Let $y = D^c f(x)$ and $y' = D^c f(x')$. Then
\begin{equation} \label{eqn:divergence.duality}
\mathbf{D}_f \left[x' : x\right] = \mathbf{D}_{f^c}^* \left[ y : y' \right].
\end{equation}
\end{corollary}

This result can be viewed as the ``biduality'' -- in the language of Zhang \cite{Z04} -- of the $c$-divergence. Namely, using the $c$-conjugate $f^c$  and the dual coordinates $y = D^c f(x)$ (representational duality) in the $L^{(\alpha)}$-divergence is equivalent to interchanging the order of the arguments (referential duality). See \cite{Z05, Z15, NZ18} for more discussions about the two dualities.

\begin{example}
Suppose $c(x, y) = \frac{1}{2} \|x - y\|^2$ is the quadratic cost on $\mathbb{R}^d$. It is easy to show that $f$ is $c$-concave if and only if $\varphi(x) = \frac{1}{2} \|x\|^2 - f(x)$ is convex. The $c$-gradient of $f$ is $D^c f = D \varphi$, and the $c$-divergence $\mathbf{D}_f$ is the Bregman divergence of $\varphi$. 
\end{example}

\subsection{A logarithmic cost function} \label{sec:log.cost}
Now we introduce a logarithmic cost which leads to the $L^{(\pm \alpha)}$-divergences. Because of the logarithm, to state the result in the framework of $c$-divergence we consider a specific domain, namely the positive quadrant $\mathcal{X} = \mathcal{Y} = \mathbb{R}_{++}^d = (0, \infty)^d$. (To make it a Polish space, we may use for example the metric $d(x, y) = ( \sum_{i = 1}^d (\log x^i - \log y^i)^2 )^{1/2}$.) We keep the discussion brief as a self-contained treatment on general convex domains will be given in Section \ref{sec:log.divergence}.

For $\alpha > 0$, consider the continuous cost function on $\mathbb{R}^d$ given by
\begin{equation} \label{eqn:c.alpha}
c^{(\alpha)}(x, y) = \frac{1}{\alpha} \log (1 + \alpha x \cdot y).
\end{equation}

\begin{remark}
While $c^{(\alpha)}$ may take negative values, it is equivalent (up to scaling and addition of linear terms, see Remark \ref{rem:cost.function.modify}) to the non-negative cost function
\begin{equation} \label{eqn:cost.reparameterized}
\log \left( \frac{1}{1 + \alpha d} + \frac{\alpha}{1 + \alpha d} \sum_{i = 1}^d x_i y_i \right) - \frac{\alpha}{1 + \alpha d} \sum_{i = 1}^d \log (x_i y_i),
\end{equation}
so that Theorem \ref{thm:ft.of.optimal.transport} is applicable. For $\alpha = 1$, this cost function was introduced and studied (under various parameterizations) in \cite{PW14, PW16, PW18}. In particular, \eqref{eqn:cost.reparameterized}  (with $\alpha = 1$) is equivalent to the excess growth rate given in \eqref{eqn:excess.growth.rate} under the change of variables $q_i = y_i/(1 + \sum_j y_j)$ and $p_i = (1/x_i) / (1 + \sum_j 1/x_j)$ for $1 \leq i \leq d$ (see \cite{PW18} for details). The symmetric representation chosen here is more convenient in information-geometric computations.

As $\alpha \rightarrow 0^+$, the cost function $c^{(\alpha)}(x, y)$ in \eqref{eqn:c.alpha} converges to $x \cdot y$. This is equivalent to the quadratic cost
\[
c(x, y) = \frac{1}{2} \|x - y\|^2 = -x \cdot y + \frac{1}{2} \|x\|^2 + \frac{1}{2} \|y\|^2
\]
after a change of variable $y \mapsto -y$ (say).
\end{remark}

The following theorem establishes the $L^{(\alpha)}$-divergence as a $c$-divergence. Analogous results hold for the $L^{(-\alpha)}$-divergence if we pick $c^{(-\alpha)} = - c^{(\alpha)}$. We omit the proof as this result will not be used in the rest of the paper and most of the work (which builds on \cite[Section 3]{PW16}) is contained in Section \ref{sec:log.divergence}.

\begin{theorem} \label{thm:L.divergence.as.c.divergence}
Consider the cost function $c = c^{(\alpha)}$ given by \eqref{eqn:c.alpha}. 
\begin{enumerate}
\item[(i)] $f$ is $c$-concave if and only if $e^{\alpha f}$ is concave on $\mathbb{R}^d$.
\item[(ii)] Suppose $f$ is $c$-concave. The $c$-gradient of $f$ at $x$, if it exists, is given by
\begin{equation} \label{eqn:c.alpha.gradient}
D^c f(x) = \frac{Df(x)}{1 - \alpha Df(x) \cdot x}.
\end{equation}
\item[(iii)] Suppose $f$ is $c$-concave and differentiable. Then the $c$-divergence of $f$ is the $L^{(\alpha)}$-divergence given by
\[
\mathbf{D}_f [ x : x' ] = \mathbf{D}^{(\alpha)}[x : x'] := \frac{1}{\alpha} \log (1 + \alpha Df(x') \cdot (x - x')) - ( f(x) - f(x')).
\]
\end{enumerate}
\end{theorem}


\section{Exponential concavity/convexity and the $L^{(\pm \alpha)}$-divergences} \label{sec:log.divergence}
In this section we define exponential concavity and convexity as well as the associated $L^{(\pm \alpha)}$-divergence $\mathbf{D}^{(\pm \alpha)}$. We also derive their self-dual representations. Our treatment is motivated by, but independent from, the connection with optimal transport maps given in Section \ref{sec:c.divergence}.

\subsection{Exponential concavity and convexity}
\begin{definition} \label{def:exponential.concavity}
Let $\Omega \subset \mathbb{R}^d$ be an open convex set, $\varphi : \Omega \rightarrow \mathbb{R}$, and fix $\alpha > 0$. Also write $\Phi = e^{\alpha \varphi}$.

\begin{enumerate}
\item[(i)] We say  that $\varphi$ is $\alpha$-exponentially concave if $\Phi$ is concave on $\Omega$. For $\alpha = 0^+$ (that is, the limit as $\alpha \rightarrow 0^+$), we say that $\varphi$ is $0^+$-exponentially concave if $\varphi$ is concave.
\item[(ii)] We say that $\varphi$ is $\alpha$-exponentially convex if $\Phi$ is convex  on $\Omega$. For $\alpha = 0^+$, we say that $\varphi$ is $0^+$-exponentially convex if $\varphi$ is convex.
\end{enumerate}
\end{definition}

Assuming $\varphi$ is twice continuously differentiable on $\Omega$, we have
\begin{equation} \label{eqn:Phi.D2}
D^2 \Phi = D^2 e^{\alpha \varphi} = \alpha e^{\alpha \varphi} \left( D^2 \varphi + \alpha (D \varphi ) (D\varphi )^{\top}\right),
\end{equation}
where $(D \varphi)^{\top}$ is the transpose of the column vector $D \varphi$. From this we have the following elementary lemma which justifies the limiting cases in Definition \ref{def:exponential.concavity}. 

\begin{lemma} \label{lem:exp.concave.convex}
Suppose $\varphi$ is twice continuously differentiable on $\Omega$ and $\alpha > 0$.
\begin{enumerate}
\item[(i)] $\varphi$ is $\alpha$-exponentially concave if and only if
\begin{equation} \label{eqn:metric.exp.concave}
-D^2 \varphi - \alpha (D \varphi) (D \varphi)^{\top} \geq 0.
\end{equation}
\item[(ii)] $\varphi$ is $\alpha$-exponentially convex if and only if
\begin{equation} \label{eqn:metric.exp.convex}
D^2 \varphi + \alpha (D \varphi) (D \varphi)^{\top} \geq 0.
\end{equation}
\end{enumerate}
Here the inequalities are in the sense of positive semidefinite matrix.
\end{lemma}

Because of the product term in \eqref{eqn:Phi.D2}, we cannot pass between exponential concavity and convexity simply by considering $-\varphi$. This is different from the classical case $\alpha = 0$. In particular, from Lemma \ref{lem:exp.concave.convex} we see that an $\alpha$-exponentially concave function is always concave, but an $\alpha$-exponentially convex function is not necessarily convex. In Section \ref{sec:L.geometry} we will regard \eqref{eqn:metric.exp.concave} and \eqref{eqn:metric.exp.convex} as the Riemannian metrics induced by $\mathbf{D}^{(\alpha)}$ and $\mathbf{D}^{(-\alpha)}$ respectively.

\subsection{$L^{(\pm \alpha)}$-divergences} \label{sec:L.alpha.duality}
By convention we always let $\alpha > 0$ be fixed. Suppose $\varphi$ is a differentiable and $\alpha$-exponentially concave function. By the concavity of $\Phi = e^{\alpha \varphi}$, for any $\xi, \xi' \in \Omega$ we have
\begin{equation} \label{eqn:L.divergence.motivation}
\Phi(\xi') + D\Phi(\xi') \cdot (\xi - \xi') \geq \Phi(\xi) \Rightarrow 1 + \alpha D \varphi(\xi') \cdot ( \xi - \xi' ) \geq e^{\alpha (\varphi(\xi) - \varphi(\xi'))}.
\end{equation}
This motivates the following definition of the $L^{(\alpha)}$-divergence. It is different from the Bregman divergence in that our divergence is given by the ratio rather than the difference in \eqref{eqn:L.divergence.motivation}.

\begin{definition} [$L^{(\alpha)}$-divergence]
If $\varphi$ is differentiable and $\alpha$-exponentially concave, we define the $L^{(\alpha)}$-divergence of $\varphi$ by
\begin{equation} \label{eqn:L.alpha.divergence}
\mathbf{D}^{(\alpha)}\left[ \xi : \xi' \right] = \frac{1}{\alpha} \log \left(1 + \alpha D \varphi(\xi') \cdot ( \xi - \xi' )\right) - \left( \varphi(\xi) - \varphi(\xi') \right), \quad \xi, \xi' \in \Omega.
\end{equation}
For $\alpha = 0^+$, we define the $L^{(0+)}$-divergence by the Bregman divergence \eqref{eqn:Bregman.divergence}.
\end{definition}

From \eqref{eqn:L.divergence.motivation} we have that $\mathbf{D}^{(\alpha)}\left[ \xi : \xi' \right] \geq 0$. If $\Phi$ is strictly concave, then $\mathbf{D}^{(\alpha)}\left[ \xi : \xi' \right] = 0$ only if $\xi = \xi'$.

For an $\alpha$-exponentially convex function, the analog of \eqref{eqn:L.divergence.motivation} is
\begin{equation} \label{eqn:L.minus.alpha.motivation}
1 + \alpha D \varphi(\xi') \cdot ( \xi - \xi' ) \leq e^{\alpha (\varphi(\xi) - \varphi(\xi'))}.
\end{equation}
Unfortunately, now the left hand side of \eqref{eqn:L.minus.alpha.motivation} may become negative and so its logarithm may not exist. Nevertheless, we can define the $L^{(-\alpha)}$-divergence locally when $\xi$ and $\xi'$ are close. This is sufficient for defining the dualistic structure $(g, \nabla, \nabla^*)$.

\begin{definition} [local $L^{(-\alpha)}$-divergence]
If $\varphi$ is differentiable and $\alpha$-exponentially convex, we define the $L^{(-\alpha)}$-divergence of $\varphi$ locally by
\begin{equation} \label{eqn:L.minus.alpha.divergence}
\mathbf{D}^{(-\alpha)}\left[ \xi : \xi' \right] = \left( \varphi(\xi) - \varphi(\xi') \right) -  \frac{1}{\alpha} \log \left(1 + \alpha D \varphi(\xi') \cdot ( \xi - \xi' )\right),
\end{equation}
for all pairs $(\xi, \xi')$ in $\Omega$ such that the logarithm in \eqref{eqn:L.minus.alpha.divergence} exists. For $\alpha = 0^+$, $\varphi$ is convex and we define (globally) the $L^{(0-)}$-divergence as the Bregman divergence
\begin{equation} \label{eqn:L.0.minus.divergence}
\mathbf{D}^{(0-)} \left[ \xi : \xi'\right] = \left( \varphi(\xi) - \varphi(\xi') \right) - D\varphi(\xi') \cdot (\xi - \xi').
\end{equation}
\end{definition}

\subsection{Duality of $L^{(\pm \alpha)}$ divergence} \label{sec:alpha.duality}
In this subsection we develop a duality theory for $\alpha$-exponentially concave and convex functions. By convention we always assume $\alpha > 0$ is fixed. We generalize the Legendre transform and conjugate which will be used to formulate the self-dual representations of the $L^{(\pm \alpha)}$-divergences. They correspond to the $c$-transform in Section \ref{sec:c.divergence}.

To ensure that the potential function $\varphi$ is well-behaved and the differential geometric objects associated with the $L^{(\pm \alpha)}$-divergences are well-defined, we will impose throughout this paper some regularity conditions.

\subsubsection{Duality for $\alpha$-exponentially concave functions}
We first consider the case where $\varphi$ is $\alpha$-exponentially concave. 

\begin{condition} [conditions for $L^{(\alpha)}$-divergence] \label{condition:concave}
We assume $\varphi$ is smooth on $\Omega$ and the matrix $-D^2 \varphi - \alpha (D \varphi) (D \varphi)^{\top}$ is strictly positive definite on $\Omega$. By considering a translation of the variable $\xi$ if necessary, we assume without loss of generality that for all $\xi \in \Omega$ we have
\begin{equation} \label{eqn:positivity}
1 - \alpha D \varphi(\xi) \cdot \xi = 1 - \alpha \sum_{ \ell = 1}^d \xi^{\ell} \frac{\partial \varphi}{\partial \xi^{\ell}}(\xi) > 0.
\end{equation}
Note that by strict concavity of $\Phi = e^{\alpha \varphi}$ and \eqref{eqn:L.divergence.motivation}, the condition \eqref{eqn:positivity} holds whenever $0 \in \overline{\Omega}$. (For example, in Section \ref{sec:log.cost} we have $\Omega = \mathbb{R}_{++}^d$ and $0 \in \overline{\Omega}$.)
\end{condition}

Observe that there does not exist any positive and strictly concave function $\Phi = e^{\alpha \varphi}$ on the real line. Thus, in order that $\varphi$ is $\alpha$-exponentially concave on $\Omega$, the domain $\Omega$ must not contain any line (rays are fine). This is different from the Bregman case where the concave function may be globally defined on $\mathbb{R}^d$.  Examples of feasible domains in our setting include the open simplex (considered in \cite{PW16}) as well as the positive quadrant $\mathbb{R}_{++}^d = (0, \infty)^d$.

Condition \eqref{eqn:positivity} will be used in the following construction. For each $\xi' \in \Omega$, the supporting tangent hyperplane of $\Phi = e^{\alpha \varphi}$ at $\xi'$ is given by
\[
\xi \in \Omega \mapsto \Phi(\xi') + D \Phi(\xi') \cdot (\xi - \xi') = e^{\alpha \varphi(\xi')} \left(1 - \alpha D \varphi(\xi') \cdot \xi' + \alpha D \varphi (\xi') \cdot \xi\right).
\]
Thus \eqref{eqn:positivity} is equivalent to positivity of the constant coefficient. By concavity of $\Phi$, for $\xi \in \Omega$ we have
\begin{equation} \label{eqn:usual.duality}
e^{\alpha \varphi(\xi)} = \min_{\xi' \in \Omega} e^{\alpha \varphi(\xi')} \left( (1 - \alpha D \varphi(\xi') \cdot \xi') + \alpha D \varphi (\xi') \cdot \xi \right).
\end{equation}
Since $\Phi$ is strictly concave, the minimum is attained uniquely at $\xi' = \xi$. This and the logarithm is the basis of Theorem \ref{prop:alpha.duality} which is our duality result. To prepare for it let us introduce some notations.

\begin{definition} [$\alpha$-gradient] 
We define the $\alpha$-gradient of $\varphi$ at $\xi \in \Omega$ by
\begin{equation} \label{eqn:alpha.gradient}
D^{(\alpha)} \varphi(\xi) = \frac{1}{1 - \alpha D \varphi(\xi) \cdot \xi } D \varphi(\xi),
\end{equation}
where $D \varphi = D^{(0)} \varphi$ is the Euclidean gradient. Condition \ref{condition:concave} ensures that it is well-defined.
\end{definition}

\begin{definition}[dual coordinate]
We define the dual coordinate $\eta$ by
\begin{equation} \label{eqn:dual.coordinate}
\eta = D^{(\alpha)} \varphi(\xi).
\end{equation}
We let $\Omega' = D^{(\alpha)} \varphi (\Omega)$ be the range of $\eta$ (which may not be convex).
\end{definition}

This terminology is justified by the following result.

\begin{proposition}
The $\alpha$-gradient $D^{(\alpha)} \varphi$ is a diffeomorphism from $\Omega$ onto $\Omega'$ which is an open set in $\mathbb{R}^d$.
\end{proposition}
\begin{proof}
It is clear from \eqref{eqn:alpha.gradient} that the map $D^{(\alpha)} \varphi$ is smooth.

To show that $D^{(\alpha)} \varphi$ is injective, suppose towards a contradiction that $D^{(\alpha)} \varphi(\xi) = D^{(\alpha)} \varphi(\xi')$ for distinct $\xi, \xi' \in \Omega$. By definition, we have
\[
\frac{D \varphi(\xi)}{1 - \alpha D \varphi(\xi) \cdot \xi} = \frac{D \varphi(\xi')}{1 - \alpha D \varphi(\xi') \cdot \xi'}.
\]
It follows that for any $\xi'' \in \Omega$, we have
\[
1 + \frac{\alpha D \varphi(\xi) \cdot \xi''}{1 - \alpha D \varphi(\xi) \cdot \xi} = 1 + \frac{\alpha D \varphi(\xi') \cdot \xi''}{1 - \alpha D \varphi(\xi') \cdot \xi'}.
\]
Rearranging, we have the identity
\[
\frac{1 + \alpha D \varphi(\xi) \cdot (\xi'' - \xi)}{1 - \alpha D \varphi(\xi) \cdot \xi} = \frac{1 + \alpha D \varphi(\xi') \cdot (\xi'' - \xi')}{1 - \alpha D \varphi(\xi') \cdot \xi'}.
\]
Letting $\xi''$ be $\xi$ and $\xi'$ respectively, we see after some computation that
\[
(1 + \alpha D \varphi(\xi) \cdot (\xi' - \xi)) \cdot \left(1 + \alpha D \varphi(\xi') \cdot (\xi - \xi')\right) = 1.
\]
Observe that this is equivalent to
\[
\mathbf{D}^{(\alpha)}\left[ \xi' : \xi \right] + \mathbf{D}^{(\alpha)}\left[ \xi : \xi' \right] = 0,
\]
which is a contradiction since $\mathbf{D}^{(\alpha)}\left[ \xi' : \xi \right], \mathbf{D}^{(\alpha)}\left[ \xi : \xi' \right] > 0$ by Condition \ref{condition:concave}. This proves that the $\alpha$-gradient is injective.

Finally, from Condition \ref{condition:concave} again we have $-D^2 \varphi - \alpha (D \varphi) (D \varphi)^{\top} > 0$. This can be used to show that the Jacobian $\frac{\partial \eta}{\partial \xi}$ is invertible. By the inverse function theorem, $D^{(\alpha)} \varphi: \Omega \rightarrow \Omega'$ is a diffeomorphism and $\Omega'$ is open. (Essentially the same result is proved in \cite[Theorem 3.2]{PW16} so we do not provide the details. Also see \eqref{eqn:Riemannian.matrix} below which gives an explicit expression of the Jacobian.)
\end{proof}

\begin{definition}[$\alpha$-conjugate] \label{def:alpha.conjugate}
Consider the function $c^{(\alpha)}$ given by \eqref{eqn:c.alpha} and defined for $ x, y \in \mathbb{R}^d$ with $1 + \alpha x \cdot y > 0$. We define the $\alpha$-conjugate $\psi$ of $\varphi$ on the open set $\Omega' = D^{(\alpha)} \varphi(\Omega)$ by
\begin{equation} \label{eqn:alpha.conjugate}
\psi(\eta) = c^{(\alpha)}(\xi, \eta) - \varphi(\xi), \quad \xi = ( D^{(\alpha)} \varphi )^{-1} (\eta).
\end{equation}
\end{definition}

By \eqref{eqn:nice.identity} below $1 + \alpha \xi \cdot \eta > 0$, hence $c^{(\alpha)}(\xi, \eta)$ and $\psi(\eta)$ are well-defined.

\begin{theorem} [$\alpha$-duality] \label{prop:alpha.duality}
For $\xi \in \Omega$, $\eta \in \Omega'$ we have
\begin{equation} \label{eqn:first.conjugate}
\varphi(\xi) = \min_{\eta' \in \Omega'} \left( c^{(\alpha)}(\xi, \eta') - \psi(\eta') \right),
\end{equation}
\begin{equation}  \label{eqn:second.conjugate}
\psi(\eta) = \min_{\xi' \in \Omega} \left( c^{(\alpha)}(\xi', \eta) - \varphi(\xi') \right).
\end{equation}
Furthermore, $\psi$ is locally $\alpha$-exponentially concave on $\Omega'$ in the sense that $-D^2 \psi - \alpha (D \psi) (D \psi)^{\top} > 0$ on $\Omega'$, and we have $1 - \alpha D \psi(\eta) \cdot \eta > 0$ on $\Omega'$ so that the $\alpha$-gradient $D^{(\alpha)} \psi(\eta) = \frac{1}{1 - \alpha D \psi(\eta) \cdot \eta} D \psi(\eta)$ is well-defined. We have
\begin{equation} \label{eqn:alpha.gradient.inverse}
( D^{(\alpha)} \varphi )^{-1} = D^{(\alpha)} \psi.
\end{equation}
\end{theorem}
\begin{proof}
Taking logarithm in \eqref{eqn:usual.duality} and rearranging, we have
\begin{equation} \label{eqn:conjugate.computation}
\begin{split}
\varphi(\xi) &= \min_{\xi' \in \Omega} \left\{ \varphi(\xi') + \frac{1}{\alpha}\log  \left( 1 - \alpha D \varphi(\xi') \cdot \xi' + \alpha D \varphi(\xi') \cdot \xi\right) \right\} \\
&= \min_{\xi' \in \Omega}  \left\{ \frac{1}{\alpha} \log \left( 1 + \alpha \frac{D \varphi(\xi')}{1 - \alpha D \varphi(\xi') \cdot \xi'} \cdot \xi\right) + \varphi(\xi') \right.\\
&\quad \quad \quad \quad \left. + \frac{1}{\alpha} \log \left(1 - \alpha D\varphi (\xi') \cdot \xi'\right)\right\} \\
&= \min_{\xi' \in \Omega} \left\{ c^{(\alpha)}(\xi, \eta') + \varphi(\xi') + \frac{1}{\alpha} \log \left(1 - \alpha D\varphi (\xi') \cdot \xi'\right) \right\}.
\end{split}
\end{equation}
Note that since $\eta' = D^{(\alpha)} \varphi(\xi')$, from \eqref{eqn:alpha.gradient} we have the useful identity
\begin{equation} \label{eqn:nice.identity}
1 + \alpha \xi' \cdot \eta' = \frac{1}{1 - \alpha D \varphi(\xi') \cdot \xi'}.
\end{equation}
With this observation, we see that
\[
\varphi(\xi') + \frac{1}{\alpha} \log \left(1 - \alpha D \varphi (\xi') \cdot \xi'\right) = - \left( c^{(\alpha)} (\xi', \eta') - \varphi(\xi')\right) = - \psi(\eta').
\]
Putting this into \eqref{eqn:conjugate.computation} gives \eqref{eqn:first.conjugate}. 

From \eqref{eqn:first.conjugate}, for any $\eta, \xi'$ we have
\[
\varphi(\xi') \leq c^{(\alpha)}(\xi', \eta) - \psi(\eta) \Rightarrow \psi(\eta) \leq c^{(\alpha)}(\xi', \eta) - \varphi(\xi').
\]
By definition we have $\psi(\eta) = c^{(\alpha)}(\xi, \eta) - \varphi(\xi)$, so \eqref{eqn:second.conjugate} holds as well. In particular, for any $\xi, \xi' \in \Omega$ we have
\begin{equation} \label{eqn:duality.minimum}
\psi(\eta) = c^{(\alpha)}(\xi, \eta) - \varphi(\xi) \leq c^{(\alpha)}(\xi', \eta) - \varphi(\xi').
\end{equation}
Using \eqref{eqn:nice.identity}, we compute
\begin{equation*}
\begin{split}
c^{(\alpha)}(\xi', \eta) - c^{(\alpha)}(\xi, \eta) &= \frac{1}{\alpha} \log \left( \frac{1 + \alpha \xi' \cdot \eta}{1 + \alpha \xi \cdot \eta} \right) \\
&= \frac{1}{\alpha} \log \left( \frac{1}{1 + \alpha \xi \cdot \eta} + \alpha \xi' \cdot \frac{\eta}{1 + \alpha \xi \cdot \eta} \right) \\
&= \frac{1}{\alpha} \log \left(1 + \alpha D \varphi(\xi) \cdot (\xi' - \xi) \right).
\end{split}
\end{equation*}
It follows that
\begin{equation} \label{eqn:c.divergence.first}
c^{(\alpha)}(\xi', \eta) - c^{(\alpha)}(\xi, \eta) - \left( \varphi(\xi') - \varphi(\xi) \right) = \mathbf{D}^{(\alpha)}\left[ \xi' : \xi \right]
\end{equation}
is the $L^{(\alpha)}$-divergence of $\varphi$. 

By \eqref{eqn:first.conjugate}, for any $\xi \in \Omega$, $\eta$ is the unique minimizer (over $\Omega'$) of
\[
\eta' \mapsto \frac{1}{\alpha} \log \left(1 + \alpha \xi \cdot \eta'\right) - \psi(\eta').
\]
Since $\Omega'$ is open, we have the first order condition $\frac{\xi}{1 + \alpha \xi \cdot \eta} = D \psi(\eta)$. Rearranging, we have 
\[
1 - \alpha D \psi(\eta) \cdot \eta = \frac{1}{1 + \alpha \xi \cdot \eta}.
\]
Thus $D^{(\alpha)} \psi(\eta)$ is well-defined, $D^{(\alpha)} \psi(\eta) = \xi$, and we have \eqref{eqn:alpha.gradient.inverse}.

It remains to show that
\begin{equation} \label{eqn:psi.D2}
-D^2 \psi(\eta) - \alpha (D \psi(\eta)) (D \psi(\eta))^{\top}
\end{equation}
is strictly positive definite on $\Omega'$. Expressing \eqref{eqn:c.divergence.first} in terms of $\psi$ using the definition \eqref{eqn:alpha.conjugate} of $\alpha$-conjugate, it can be shown that
\[
\mathbf{D}^{(\alpha)}\left[ \xi : \xi' \right] = \frac{1}{\alpha} \log \left(1 + \alpha D \psi( \eta) \cdot (\eta' - \eta)\right) - \left( \psi(\eta') - \psi(\eta) \right),
\]
so the $L^{(\alpha)}$-divergence of $\varphi$ is the $L^{(\alpha)}$-divergence of $\psi$ with the arguments interchanged.  (This is a special case of Corollary \ref{cor:c.divergence.duality}.) It follows that \eqref{eqn:psi.D2} is the Riemannian metric of $\mathbf{D}^{(\alpha)}$ expressed in terms of the dual coordinates $\eta$. Since $\varphi$ satisfies Condition \ref{condition:concave} by assumption, the matrix \eqref{eqn:psi.D2} is strictly positive definite as well.
\end{proof}

From \eqref{eqn:alpha.conjugate} and \eqref{eqn:c.divergence.first}, we immediately obtain the self-dual representation which completes the circle of ideas and the duality theory for $\alpha$-exponentially concave functions. 

\begin{theorem}[self-dual representation] \label{thm:alpha.duality}
The $L^{(\alpha)}$-divergence of $\varphi$ admits the self-dual representation
\begin{equation} \label{eqn:alpha.self.dual}
\mathbf{D}^{(\alpha)}\left[ \xi : \xi' \right] = c^{(\alpha)}(\xi, \eta') - \varphi(\xi) - \psi(\eta'), \quad \xi, \xi' \in \Omega.
\end{equation}
\end{theorem}

\subsubsection{Duality for $\alpha$-exponentially convex functions}
Now we state without proof the analogous results for the $L^{(-\alpha)}$-divergence defined by an $\alpha$-exponentially convex function. The only difference is that because $e^{\alpha \varphi}$ is convex, we have to take maximums instead of minimums and switch the inequalities accordingly. Also, as the $L^{(-\alpha)}$-divergence may not be globally defined, here we only consider the local geometry.

\begin{condition} [conditions for $L^{(-\alpha)}$-divergence]
We assume $\varphi$ is smooth and the matrix $D^2 \varphi + \alpha (D \varphi) (D \varphi)^{\top}$ is strictly positive definite on $\Omega$. Translating and restricting the domain if necessary, we assume that $0 \in \Omega$ and
\begin{equation}
1 + \alpha D\varphi(\xi') \cdot (\xi  - \xi') > 0
\end{equation}
for all $\xi, \xi' \in \Omega$. This implies that the $L^{(-\alpha)}$-divergence is defined for all $\xi, \xi' \in \Omega$.
\end{condition}

\begin{theorem} [duality for $L^{(-\alpha)}$-divergence] \label{thm:duality.minus.alpha}
Consider the $\alpha$-gradient and $c^{(\alpha)}$ defined as in \eqref{eqn:alpha.gradient} and \eqref{eqn:c.alpha}. Then the map $\xi \mapsto \eta = D^{(\alpha)} \varphi(\xi) = \frac{ D\varphi(\xi)}{1 - \alpha \varphi(\xi) \cdot \xi}$ is a diffeomorphism, and the $(-\alpha)$-conjugate $\psi$ defined by \eqref{eqn:alpha.conjugate} is locally $\alpha$-exponentially convex. The functions $\varphi$ and $\psi$ are related by
\[
\varphi(\xi) = \max_{\eta' \in \Omega'} \left(c^{(\alpha)}(\xi, \eta') - \psi(\eta')\right), \quad \psi(\eta) = \max_{\xi' \in \Omega} \left(c^{(\alpha)}(\xi', \eta) - \varphi(\xi')\right).
\]
The $L^{(-\alpha)}$-divergence \eqref{eqn:L.minus.alpha.divergence} admits the self-dual representation
\begin{equation} \label{eqn:alpha.self.dual}
\mathbf{D}^{(-\alpha)}\left[ \xi : \xi' \right] =  \varphi(\xi) + \psi(\eta') - c^{(\alpha)}(\xi, \eta'), \quad \xi, \xi' \in \Omega.
\end{equation}

\end{theorem}

\section{Generalizations of exponential family and R\'{e}nyi divergence} \label{sec:q.exponential}
It is well known that Bregman divergences ($\mathbf{D}^{(0\pm)}$) arise naturally in applications involving exponential families. In fact, it was shown by Banerjee et al in \cite{BMDG05} that there is a one-to-one correspondence between what they call regular Bregman divergences and regular exponential families. Are there parameterized families of probability distributions that correspond to the $L^{(\pm \alpha)}$ divergences? 

In this section we consider generalizations of exponential family that are closely related to the $\alpha$-family \cite[Section 2.6]{AN02} and the $q$-exponential family \cite[Section 4.3]{A16}. We show that a suitably defined potential function is $\alpha$-exponentially concave/convex, and the corresponding $L^{(\pm \alpha)}$-divergences are R\'{e}nyi divergences. This gives a natural framework to study the geometry induced by the R\'{e}nyi divergence.

\subsection{Motivations} \label{sec:F.family.motivation}
We first recall the main idea in \cite{BMDG05}. Consider an exponential family of probability densities of the form
\begin{equation} \label{eqn:exp.family}
\log p(x, \xi) = \xi \cdot h(x) - \varphi(\xi),
\end{equation}
where $h(x) = (h^1(x), \ldots, h^d(x))$ is a vector of sufficient statistics, and $\varphi$ is the convex cumulant generating function. Let $\mathbf{D} = \mathbf{D}^{(0-)}$ be the Bregman divergence of $\varphi$. By the self-dual representation (induced by the function $c(\xi, h(x)) = - \xi \cdot h(x)$; see Section \ref{sec:c.divergence}), we have
\begin{equation} \label{eqn:exp.family2}
\log p(x, \xi) = - \mathbf{D}\left[ \xi : \xi' \right] + \psi(h(x)),
\end{equation}
where $\psi$ is the convex conjugate of $\varphi$ and $h(x)$ plays the role of the dual variable $\eta'$. Remarkably, the function $-\psi$ is the Shannon entropy (see \cite[Section 2.1]{A16}). 

Now let us consider an $L^{(\alpha)}$-divergence $\mathbf{D}^{(\alpha)}$ with  $\alpha > 0$, where $\varphi$ is $\alpha$-exponentially concave. Its self-dual representation (see Theorem \ref{thm:alpha.duality}) is
\[
\mathbf{D}^{(\alpha)}\left[ \xi : \xi' \right] = \frac{1}{\alpha} \log \left(1 + \alpha \xi \cdot \eta'\right) - \varphi(\xi) - \psi(\eta').
\]
Imitating \eqref{eqn:exp.family2}, let us consider a parameterized density of the form
\begin{equation} \label{eqn:new.family.motivation}
\log p(x, \xi) =  - \mathbf{D}^{(\alpha)}\left[ \xi : \xi' \right] - \psi(h(x)) = -\frac{1}{\alpha} \log (1 + \alpha \xi \cdot h(x)) + \varphi(\xi),
\end{equation}
where $h(x) = \eta'$ as well. Rearranging gives the representation
\begin{equation} \label{eqn:new.density}
p(x, \xi) = \left(1 + \alpha \xi \cdot h(x) \right)^{-\frac{1}{\alpha}} e^{\varphi(\xi)}.
\end{equation}
If $\mu$ is the dominating measure of the family, the normalizing function $\varphi(\xi)$ in \eqref{eqn:new.density} is given by
\begin{equation} \label{eqn:new.cgf}
\varphi(\xi) = - \log \int (1 + \alpha \xi \cdot h(x))^{-\frac{1}{\alpha}} d\mu(x).
\end{equation}
It is the analogue of the cumulant generating function in \eqref{eqn:exp.family}. For lack of a better name, let us call \eqref{eqn:new.density} an $\mathcal{F}^{(\alpha)}$-family.

Similarly, the $L^{(-\alpha)}$-divergence leads to the $\mathcal{F}^{(-\alpha)}$-family
\begin{equation} \label{eqn:new.density2}
p(x, \xi) = \left(1 + \alpha \xi \cdot h(x) \right)^{\frac{1}{\alpha}} e^{-\varphi(\xi)},
\end{equation}
where 
\begin{equation} \label{eqn:new.cgf2}
\varphi(\xi) = \log \int (1 + \alpha \xi \cdot h(x))^{\frac{1}{\alpha}} d\mu(x).
\end{equation}

\begin{remark}
For $\alpha' \neq 1$, Amari and Nagaoka \cite[Section 2.6]{AN02} defined an $\alpha'$-family of positive measures whose densities are given by
\begin{equation} \label{eqn:alpha.family.density}
\begin{split}
& L^{(\alpha)}(p(x, \xi)) = C(x) + \sum_i \xi^i F^i(x), \quad L^{(\alpha)}(u) = \frac{2}{1 - \alpha'} u^{\frac{1 - \alpha'}{2}},\\
& \text{or} \quad p(x, \xi) = \left[\frac{1 - \alpha'}{2} \left( C(x) + \sum_i \xi^i F^i(x) \right)\right]^{\frac{2}{1 - \alpha'}},
\end{split}
\end{equation}
where $C$ and $F^i$ are functions on $\mathcal{X}$. When $\alpha' \rightarrow 1$ then \eqref{eqn:alpha.family.density} reduces to the exponential family. Comparing this with \eqref{eqn:new.cgf} and \eqref{eqn:new.cgf2}, we see that our $\mathcal{F}^{(\pm \alpha)}$-families are essentially a normalized version of their $\alpha'$-family with a different parameterization. We keep our formulation so as to be consistent with \eqref{eqn:new.family.motivation} (which is the motivation of this section) and the framework of the paper.
\end{remark}

\subsection{Exponential concavity/convexity of $\varphi$}
In order that the densities \eqref{eqn:new.density} and \eqref{eqn:new.density2} are well-defined, we need to check that the potential functions \eqref{eqn:new.cgf} and \eqref{eqn:new.cgf2} indeed define $\alpha$-exponentially concave/convex functions. We impose the following conditions which formalize the ideas in Section \ref{sec:F.family.motivation}.

\begin{condition}
Let $\mu$ be a probability measure on a measurable space $\mathcal{X}$. Let $h = (h^1, \ldots, h^d) : \mathcal{X} \rightarrow \mathbb{R}_{++}^n$ be a vector-valued function such that if $\xi \cdot h(x) = \xi' \cdot h(x)$ $\mu$-almost everywhere on $\mathcal{X}$ then $\xi = \xi'$.

For the $\mathcal{F}^{(\alpha)}$-case, we assume that $1 + \alpha \xi \cdot h > 0$ and the integral in \eqref{eqn:new.cgf} is finite for all $\xi$ in an open convex set $\Omega \subset \mathbb{R}^d$ and we can differentiate with respect to $\xi$ under the integral sign. Analogous conditions are imposed for the $\mathcal{F}^{(-\alpha)}$-case. In both cases we call $\varphi$ the potential function of the family.
\end{condition}

\subsubsection{$\mathcal{F}^{(\alpha)}$-family}

\begin{proposition} \label{prop:check.exp.concave}
For $\alpha > 0$ fixed, the potential function $\varphi(\xi)$ of the $\mathcal{F}^{(\alpha)}$-family defined by \eqref{eqn:new.cgf} is $\alpha$-exponentially concave on $\Omega$.
\end{proposition}
\begin{proof}
We need to check that $\Phi = e^{\alpha \varphi}$ is concave on $\Omega$. From \eqref{eqn:new.cgf}, we have
\begin{equation} \label{eqn:new.Phi}
\Phi(\xi) = \left( \int (1 + \alpha \xi \cdot h(x) )^{-\frac{1}{\alpha}} d\mu(x) \right)^{-\alpha},
\end{equation}
which can be regarded as a generalized harmonic average.

Let $X$ be the random element $X(x) = x$, where $x \in \mathcal{X}$ and $\mathcal{X}$ is the sample space. Given $\xi \in \Omega$, let $\mathbb{E}_{\xi}$ denote the expectation with respect to $p(x, \xi)$, and let $Z_{\xi}$ be the random vector defined by
\begin{equation} \label{eqn:new.sufficient}
Z_{\xi} = \frac{h(X)}{1 + \alpha \xi \cdot h(X)}.
\end{equation}
Differentiating \eqref{eqn:new.Phi} and simplifying, we have
\begin{equation*}
\begin{split}
\frac{\partial \Phi}{\partial \xi^i}(\xi) &= \alpha \left( \int (1 + \alpha \xi \cdot h(x))^{-\frac{1}{\alpha}} d\mu \right)^{-\alpha - 1} \left( \int (1 + \alpha \xi \cdot h(x) )^{-\frac{1}{\alpha} - 1} h^i(x)d \mu\right) \\
&= \alpha \left( \int (1 + \alpha \xi \cdot h(x))^{-\frac{1}{\alpha}} d\mu(x) \right)^{-\alpha} \int p(x, \xi) Z_{\xi}^i d\mu(x) \\
&= \alpha \Phi(\xi) \mathbb{E}_{\xi} Z_{\xi}^i.
\end{split}
\end{equation*}
In particular, we have
\begin{equation} \label{eqn:new.varphi.gradient}
D \varphi(\xi) = \mathbb{E}_{\xi} [Z_{\xi}]
\end{equation}
which can be interpreted as an expectation parameter (although $Z_{\xi}$ is not observable if $\xi$ is unknown). Differentiating one more time and calculating carefully, we get
\begin{equation*}
\begin{split}
\frac{\partial^2 \Phi}{\partial \xi^i \partial \xi^j} (\xi) &= \alpha (1 + \alpha) \left( \mathbb{E}_{\xi} [Z_{\xi}^i ] \mathbb{E}_{\xi}  [Z_{\xi}^j] - \mathbb{E}_{\xi}  [Z_{\xi}^i  Z_{\xi}^j] \right) \\
&= -\alpha (1 + \alpha) \mathrm{Cov} (  Z_{\xi}^i ,  Z_{\xi}^j ).
\end{split}
\end{equation*}
Since $D^2 \Phi$ is a negative multiple of the covariance matrix of $Z_{\xi}$,
 $\Phi$ is concave and $\varphi$ is $\alpha$-exponentially concave.
\end{proof}

\subsubsection{$\mathcal{F}^{(-\alpha)}$-family}
Next we consider the $\mathcal{F}^{(-\alpha)}$-family. Interestingly, the result is a little different and depends on the value of $\alpha$.

\begin{proposition} \label{prop:check.exp.concave2}
Consider the potential function $\varphi(\xi)$ of the $\mathcal{F}^{(-\alpha)}$-family defined by \eqref{eqn:new.cgf2}.
\begin{enumerate}
\item If $0 < \alpha < 1$, then $\varphi$ is $\alpha$-exponentially convex.
\item If $\alpha > 1$, then $\varphi$ is $\alpha$-exponentially concave.
\end{enumerate}
\end{proposition}
\begin{proof}
Now we heave
\[
\Phi(\xi) = e^{\alpha \varphi(\xi)} = \left( \int (1 + \alpha \xi \cdot h )^{\frac{1}{\alpha}} d\mu \right)^{\alpha}.
\]
A computation similar to that in the proof of Proposition \ref{prop:check.exp.concave} shows that
\[
D\varphi(\xi) = \mathbb{E}_{\xi} Z_{\xi} = \mathbb{E}_{\xi} \left[ \frac{h(X)}{1 + \alpha \xi \cdot h(X)} \right]
\]
and
\[
\frac{\partial^2 \Phi}{\partial \xi^i \partial \xi^j}(\xi) = \alpha (1 - \alpha) \mathrm{Cov}_{\xi}(Z_{\xi}^i, Z_{\xi}^j).
\]
Now $\Phi$ is convex if $0 < \alpha < 1$ and concave if $\alpha > 1$. 

\end{proof}

\subsection{R\'{e}nyi entropy and divergence} \label{sec:Renyi.family}
Based on information-theoretic reasonings, R\'{e}nyi entropy and divergence were first introduced by R\'{e}nyi in \cite{R61} and they have numerous applications in information theory and statistical physics. See \cite{VH14} for a survey of their remarkable properties. We show that these quantities arise naturally in our $\mathcal{F}^{(\pm \alpha)}$-families.

We first recall the definitions of R\'{e}nyi entropy and divergence. Same as above we fix a measure space $(\mathcal{X}, \mu)$.

\begin{definition}
Let $\tilde{\alpha} \in (0, 1) \cup (1, \infty)$. Let $P$, $Q$ be probability measures on $\mathcal{X}$ that are absolutely continuous with respect to $\mu$. We let $p = \frac{dP}{d\mu}$ and $q = \frac{dQ}{d\mu}$ be their densities.
\begin{enumerate}
\item The R\'{e}nyi entropy of $P$ of order $\tilde{\alpha}$ (and with reference measure $\mu$) is defined by
\begin{equation} \label{eqn:Renyi.entropy}
\mathbf{H}_{\tilde{\alpha}} (P) = \frac{1}{1 - \tilde{\alpha}} \log \int p^{\tilde{\alpha}} d\mu.
\end{equation}
We also denote the above by $\mathbf{H}_{\tilde{\alpha}} (p)$.
\item The R\'{e}nyi divergence or order $\tilde{\alpha}$ between $P$ and $Q$ is defined by
\begin{equation} \label{eqn:Renyi.divergence}
\mathbf{D}_{\tilde{\alpha}} \left( P || Q \right) = \frac{1}{\tilde{\alpha} - 1} \log \int p^{\tilde{\alpha}} q^{1 - \tilde{\alpha}} d\mu.
\end{equation}
\end{enumerate}
\end{definition}

Note that $\mathbf{H}_{\tilde{\alpha}}(P) = - \mathbf{D}_{\tilde{\alpha}} \left( P || \mu\right)$ if $\mu$ itself is a probability measure. We use a different notation to distinguish the R\'{e}nyi divergence from the $L^{(\pm \alpha)}$-divergences $\mathbf{D}^{(\pm \alpha)}\left[ \xi : \xi'\right]$ (and $\tilde{\alpha}$ from $\alpha$). It is well known that as $\tilde{\alpha} \rightarrow 1$, the {R\'{e}nyi entropy and divergence converge respectively to the Shannon entropy and the Kullback-Leibler divergence.

The following is the main result of this section. The different cases correspond to the exponential concavity/convexity of $\varphi$ established in Proposition \ref{prop:check.exp.concave} and Proposition \ref{prop:check.exp.concave2}.

\begin{theorem} \label{thm:renyi} { \ }
\begin{enumerate}
\item[(i)] Consider an $\mathcal{F}^{(\alpha)}$-family with $\alpha > 0$ and potential function $\varphi$ (see \eqref{eqn:new.cgf}). Then the $L^{(\alpha)}$-divergence of $\varphi$ is the R\'{e}nyi entropy of order $\tilde{\alpha} := 1 + \alpha$:
\begin{equation} \label{eqn:new.entropy}
\mathbf{D}^{(\alpha)}\left[ \xi : \xi' \right] = \mathbf{D}_{\tilde{\alpha}} \left( p(\cdot , \xi') || p(\cdot, \xi) \right).
\end{equation}
Moreover, the $\alpha$-conjugate $\psi$ of $\varphi$ (see Definition \ref{def:alpha.conjugate}) is the R\'{e}nyi entropy of order $\tilde{\alpha}$:
\begin{equation} \label{eqn:conjugate.renyi}
\psi(\eta) = \mathbf{H}_{\tilde{\alpha}} ( p( \cdot, \xi) ), \quad \eta = D^{(\alpha)} \varphi(\xi).
\end{equation}
\item[(ii)] Consider an $\mathcal{F}^{(-\alpha)}$-family (see \eqref{eqn:new.density2}). If $0 < \alpha < 1$, the $L^{(-\alpha)}$-divergence of $\varphi$ is the R\'{e}nyi entropy of order $\tilde{\alpha} := 1 - \alpha$:
\begin{equation} \label{eqn:new.entropy2}
\mathbf{D}^{(-\alpha)}\left[ \xi : \xi' \right] = \mathbf{D}_{\tilde{\alpha}} \left( p(\cdot , \xi') || p(\cdot, \xi) \right),
\end{equation}
and the $(-\alpha)$-conjugate is $\psi(\eta) =  -\mathbf{H}_{\tilde{\alpha}} ( p( \cdot, \xi) )$ (note the minus sign).

If $\alpha > 1$, the $L^{(\alpha)}$-divergence of $\varphi$ is given by
\begin{equation} \label{eqn:new.entropy3}
\mathbf{D}^{(\alpha)}\left[ \xi : \xi' \right] = \frac{\alpha - 1}{\alpha} \mathbf{D}_{\alpha} \left( p(\cdot , \xi) || p(\cdot, \xi') \right),
\end{equation}
Here $\tilde{\alpha} = \alpha$ and the order of the arguments stay the same. The $(-\alpha)$-conjugate is
\[
\psi(\eta) = \frac{-1}{\alpha} \log \int p(x, \xi)^{1 - \alpha} d\mu.
\]
Note that it is not a R\'{e}nyi entropy since $1 - \alpha < 0$.
\end{enumerate}
\end{theorem}
\begin{proof}
(i) We want to compute
\[
\mathbf{D}^{(\alpha)}\left[ \xi : \xi' \right] = \frac{1}{\alpha} \log \left( 1 + \alpha D \varphi(\xi') \cdot (\xi - \xi')\right) - (\varphi(\xi) - \varphi(\xi')).
\]
First, using \eqref{eqn:new.sufficient} and \eqref{eqn:new.varphi.gradient}, we have
\begin{equation} \label{eqn:Renyi.compute}
\begin{split}
1 + \alpha D \varphi(\xi') \cdot (\xi - \xi') &= 1 + \alpha \mathbb{E}_{\xi'} [Z_{\xi'} ] \cdot (\xi - \xi') \\
  &= 1 + \alpha \mathbb{E}_{\xi'} \left[ \frac{h(X)}{1 + \alpha \xi' \cdot h(X)} \cdot (\xi - \xi') \right] \\
  &= \mathbb{E}_{\xi'} \frac{1 + \alpha \xi \cdot h(X)}{1 + \alpha \xi' \cdot h(X)}.
\end{split}
\end{equation}
From \eqref{eqn:new.density}, we have 
\[
1 + \alpha \xi' \cdot h(X) = \frac{e^{\alpha \varphi(\xi')}}{p(X, \xi')^{\alpha}}.
\]
It follows from \eqref{eqn:Renyi.compute} that
\begin{equation*}
\begin{split}
\mathbf{D}^{(\alpha)}\left[ \xi : \xi' \right] &= \frac{1}{\alpha} \log \left( 1 + \alpha D \varphi(\xi') \cdot (\xi - \xi')\right) - (\varphi(\xi) - \varphi(\xi')) \\
&= \frac{1}{\alpha} \log \mathbb{E}_{\xi'} \frac{ e^{\alpha \varphi(\xi)}/ p(X, \xi)^{\alpha}}{ e^{\alpha \varphi(\xi')}/ p(X, \xi')^{\alpha}} - (\varphi(\xi) - \varphi(\xi')) \\
&= \frac{1}{\alpha} \log \mathbb{E}_{\xi'} \frac{p(X, \xi')^{\alpha}}{p(X, \xi)^{\alpha}}.
\end{split}
\end{equation*}
Expressing the divergence in terms of the integral, we have
\begin{equation*}
\begin{split}
\mathbf{D}^{(\alpha)}\left[ \xi : \xi' \right] &= \frac{1}{\alpha} \log \int p(x, \xi')^{1 + \alpha} p(x, \xi)^{-\alpha} d\mu \\
&= \frac{1}{\tilde{\alpha} - 1} \log \int p(x, \xi')^{\tilde{\alpha}} p(x, \xi)^{1 - \tilde{\alpha}} d\mu \\
&= \mathbf{D}_{\tilde{\alpha}}\left( p(\cdot, \xi') \mid \mid p(\cdot, \xi) \right),
\end{split}
\end{equation*}
where $\tilde{\alpha} = 1 + \alpha > 1$. 

Next we consider the $\alpha$-conjugate
\[
\psi(\eta) = \frac{1}{\alpha} \log \left(1 + \alpha \xi \cdot \eta\right) - \varphi(\xi).
\]
Using the identity \eqref{eqn:nice.identity} and following the argument in \eqref{eqn:Renyi.compute}, we have
\begin{equation*}
\begin{split}
1 + \alpha \xi \cdot \eta &= \left(1 - \alpha D \varphi(\xi) \cdot \xi\right)^{-1} \\
  &= \left( \mathbb{E}_{\xi} \left[ (1 + \alpha \xi \cdot h(X))^{-1} \right] \right)^{-1} \\
  &= \left( \mathbb{E}_{\xi} \left[ p(X, \xi)^{\alpha} e^{-\alpha \varphi(\xi)} \right]\right)^{-1}.
\end{split}
\end{equation*}
From this we get
\begin{equation*}
\begin{split}
\psi(\eta) &= -\frac{1}{\alpha} \log \int p(x, \xi)^{1 + \alpha} e^{-\alpha \varphi(\xi)} d\mu - \varphi(\xi) \\
  &= \frac{1}{1 - \tilde{\alpha}} \int p(x, \xi)^{\tilde{\alpha}} d\mu = \mathbf{H}_{\tilde{\alpha}}(p(\cdot, \xi)).
\end{split}
\end{equation*}
This proves (i), and the proof of (ii) is similar.
\end{proof}

\begin{remark}
It appears from the above results that the $\mathcal{F}^{(-\alpha)}$-family with $\alpha > 1$ is different from the other cases. Nevertheless, the $\mathcal{F}^{(\alpha)}$-family for $\alpha > 0$ and the $\mathcal{F}^{(-\alpha)}$-family for $0 < \alpha < 1$ together give all R\'{e}nyi entropy/divergence of order $\tilde{\alpha} \in (0, 1) \cup (1, \infty)$.
\end{remark}

\begin{example}[finite simplex] \label{eg:simplex}
Let $\mathcal{X} = \{0, 1, \ldots, d\}$ be a finite set and let $\mu$ be the counting measure on $\mathcal{X}$. Fix $\alpha > 0$. Let $h^i(x) = \delta^i(x)$ be the indicator function for point $i$, $i = 1, \ldots, d$. Then \eqref{eqn:new.density} gives, for $\alpha > 0$, the $\mathcal{F}^{(\alpha)}$-family of probability mass functions on $\mathcal{X}$ given by
\begin{equation} \label{eqn:discrete.alpha.exponential.family}
p(i, \xi) = (1 + \alpha \xi^i)^{-1/\alpha} e^{\varphi(\xi)}, \quad i = 1, \ldots, d, \quad p(0, \xi) = e^{\varphi(\xi)},
\end{equation}
where the $\alpha$-exponentially concave function $\varphi$ is
\[
\varphi(\xi) = \log p(0, \xi) = -\log \left(1 + \sum_{i = 1}^d (1 + \alpha \xi^i)^{-1/\alpha}\right),
\]
and the domain is $\Omega = \{\xi \in \mathbb{R}^d: 1 + \alpha \xi^i > 0 \; \forall i\}$. The coordinate $\xi \in \Omega$ is a global coordinate system of the open simplex $\mathcal{S}_d = \{p = (p^0, \ldots, p^d) \in (0, 1)^{1+d}: \sum_i p^i = 1\}$. From \eqref{eqn:discrete.alpha.exponential.family}, we have
\[
\xi^i = \frac{1}{\alpha} \left( \left( \frac{p(0, \xi)}{p(i, \xi)}\right)^{\alpha} - 1 \right), \quad i = 1, \ldots, d.
\]
By Theorem \ref{thm:renyi}, the $L^{(\alpha)}$-divergence of $\varphi$ is the discrete R\'{e}nyi  divergence of order $\tilde{\alpha} = 1 + \alpha > 0$. 

Alternatively, we may express the simplex $\mathcal{S}_d$ as an $\mathcal{F}^{(-\alpha)}$-family where
\[
p(i, \xi) = (1 + \alpha \xi^i)^{\frac{1}{\alpha}} e^{-\varphi(\xi)}, \quad i = 1, \ldots, d, \quad p(0, \xi) = e^{-\varphi(\xi)},
\]
\[
\varphi(\xi) = -\log p(0, \xi) = \log \left(1 + \sum_{j = 1}^d (1 + \alpha \xi^j)^{1/\alpha}\right),
\]
and
\[
\xi^i = \frac{1}{\alpha} \left( \left( \frac{p(i, \xi)}{p(0, \xi)}\right)^{\alpha} - 1 \right), \quad i = 1, \ldots, d.
\]
When $0 < \alpha < 1$, the $L^{(-\alpha)}$-divergence of $\varphi$ is the discrete R\'{e}nyi  divergence of order $\tilde{\alpha} = 1 - \alpha > 0$. 

The corresponding dualistic structures will be considered in Section \ref{sec:alpha.divergence}.
\end{example}

\section{Dualistic structure of $L^{(\pm \alpha)}$-divergence} \label{sec:L.geometry}
In this and the next sections we study the dualistic structure induced by a given $L^{(\pm \alpha)}$-divergence. First we state the general definition of dualistic structure. For preliminaries in differential geometry see \cite{A16, AJLS17}.

\begin{definition} \label{def:dualistic.structure}
Let $M$ be a smooth manifold. A dualistic structure on $\mathcal{M}$ is a quadruplet $(M, g, \nabla, \nabla^*)$ where $g$ (also denoted by $\langle \cdot, \cdot \rangle$) is a Riemannian metric, and $(\nabla, \nabla^*)$ is a pair of torsion-free affine connections that are dual with respect to $g$: for any vector fields $X$, $Y$ and $Z$, the covariant derivatives satisfy the identity
\begin{equation} \label{eqn:connection.duality}
Z \langle X, Y \rangle = \langle \nabla_Z X, Y \rangle + \langle X, \nabla_Z^* Y \rangle.
\end{equation}
We also call $\mathcal{M}$ equipped with a dualistic structure a statistical manifold.
\end{definition}

For a dualistic structure, it is easy to see that the average $\overline{\nabla} = \frac{1}{2} (\nabla + \nabla^*)$ is the Levi-Civita connection of the metric $g$. By the general results of Eguchi (see \cite{E83, E92}), any divergence (see \cite[Definition 1.1]{A16}) on a manifold $\mathcal{M}$ defines a dualistic structure on $\mathcal{M}$. It is this structure induced by $\mathbf{D}^{(\pm \alpha)}$ that we study in this section.

We will derive explicit coordinate representations of this dualistic structure. The geometric consequences will be studied in Section \ref{sec:geometric.consequences}. We stress that while any divergence induces a geometric structure, the self-dual representations motivated by optimal transport suggest the appropriate coordinate systems $\xi$ and $\eta = D^{(\alpha)} \varphi(\xi)$.

By the results of Section \ref{sec:q.exponential}, for an $\mathcal{F}^{(\pm \alpha)}$-family this gives the geometry induced by the R\'{e}nyi divergence. The R\'{e}nyi divergence is closely related to the $\alpha$-divergence.\footnote{The author thanks Shun-ichi Amari for pointing this out.} Our framework is more general as it does not depend on the probabilistic representation. We discuss the connection with the $\alpha$-divergences in Section \ref{sec:alpha.divergence}.

\subsection{Notations}
We focus on the case of $L^{(\alpha)}$-divergence. The arguments for the $L^{(-\alpha)}$-divergence are essentially the same (with some changes of signs due to the different self-dual representation) and will be left for the reader. The main results for the $L^{(-\alpha)}$-divergence will be stated in Section \ref{sec:geometry.minus.alpha}.

Henceforth we fix an $\alpha$-exponentially concave function $\varphi$ on $\Omega$ satisfying Condition \ref{condition:concave}, and let $\mathbf{D} = \mathbf{D}^{(\alpha)}$ be the $L^{(\alpha)}$-divergence of $\varphi$. We will use the self-dual representation derived in Theorem \ref{thm:alpha.duality}:
\begin{equation} \label{eqn:divergence.self.dual}
\mathbf{D}\left[ \xi : \xi' \right] = \frac{1}{\alpha} \log \left(1 + \alpha \xi \cdot \eta' \right) - \varphi(\xi) - \psi(\eta').
\end{equation}
Here $\eta = D^{(\alpha)} \varphi(\xi)$ is the dual coordinate system given by the $\alpha$-gradient. To simplify the notations we let
\begin{equation} \label{eqn:notation.Pi}
\Pi(\xi, \eta') = \Pi^{(\alpha)} (\xi, \eta') = 1 + \alpha \xi \cdot \eta',
\end{equation}
where $\xi \in \Omega$ and $\eta' \in \Omega' = D^{(\alpha)} \varphi(\Omega)$. We regard $\xi$ and $\eta$ as column vectors. The transpose of a vector or matrix $A$ is denoted by $A^{\top}$. The identity matrix is denoted by $I$.

Consider the smooth manifold $\mathcal{M} = \Omega$ where the primal coordinate $\xi$ (identity map) and the dual coordinate $\eta = D^{(\alpha)} \varphi(\xi)$ are global coordinate systems. Note that by symmetry of the cost function $c^{(\alpha)}$ and the associated $\alpha$-duality, it suffices to consider the primal connection $\nabla$ expressed in the primal coordinate system $\xi$. Consider the coordinate vector fields $\frac{\partial}{\partial \xi^1}, \ldots, \frac{\partial}{\partial \xi^d}$. By definition, the coordinate representations of $(g, \nabla, \nabla^*)$ under $\xi$ are given by
\begin{equation} \label{eqn:geometry.coefficients}
\begin{split}
g_{ij}(\xi) &= \left\langle \frac{\partial}{\partial \xi^i}, \frac{\partial}{\partial \xi^j} \right\rangle =  - \left. \frac{\partial}{\partial \xi^i} \frac{\partial}{\partial \xi^{'j}} \mathbf{D}\left[ \xi : \xi' \right] \right|_{\xi = \xi'}, \\
\Gamma_{ijk}(\xi) &= \left\langle \nabla_{\frac{\partial}{\partial \xi^i}} \frac{\partial}{\partial \xi^j}, \frac{\partial}{\partial \xi^k} \right\rangle = -\left. \frac{\partial^2}{\partial \xi^i \partial \xi^j} \frac{\partial}{\partial \xi^{'k}} \mathbf{D}\left[ \xi : \xi' \right] \right|_{\xi = \xi'}, \\
\Gamma_{ijk}^*(\xi) &= \left\langle \nabla_{\frac{\partial}{\partial \xi^i}}^* \frac{\partial}{\partial \xi^j}, \frac{\partial}{\partial \xi^k} \right\rangle = - \left. \frac{\partial^2}{\partial \xi^{'i} \partial \xi^{'j}} \frac{\partial}{\partial \xi^k} \mathbf{D}\left[ \xi : \xi' \right] \right|_{\xi = \xi'}.
\end{split}
\end{equation}
Note that by construction the duality \eqref{eqn:connection.duality} automatically holds (see for example \cite[Theorem 6.2]{A16}). We also define
\begin{equation} \label{eqn:Christoffel.symbols}
\begin{split}
{\Gamma_{ij}}^k(\xi) &= \Gamma_{ijm}(\xi) g^{mk} (\xi), \\
{\Gamma_{ij}^*}\mathstrut^{k}(\xi) &= \Gamma_{ijm}^*(\xi) g^{mk}(\xi),
\end{split}
\end{equation}
where $\left( g^{ij}(\xi) \right)$ is the inverse of $\left( g_{ij}(\xi) \right)$ and the Einstein summation convention (see for example \cite[p.21]{A16}) is used. This implies that
\[
 \nabla_{\frac{\partial}{\partial \xi^i}} \frac{\partial}{\partial \xi^j} = {\Gamma_{ij}}^k(\xi)  \frac{\partial}{\partial \xi^k}, \quad  \nabla_{\frac{\partial}{\partial \xi^i}}^* \frac{\partial}{\partial \xi^j} = {\Gamma_{ij}^*}\mathstrut^{k}(\xi)  \frac{\partial}{\partial \xi^k}.
\]
We also let $\frac{\partial \xi}{\partial \eta} = \left( \frac{\partial \xi^i}{\partial \eta^j} \right)$ and  $\frac{\partial \eta}{\partial \xi} = \left( \frac{\partial \eta^i}{\partial \xi^j} \right)$ be the Jacobian matrices of the transition maps. They are inverses of each other.

\subsection{Dualistic structure of $L^{(\alpha)}$-divergence}
\subsubsection{The Riemannian metric}
\begin{proposition} \label{prop:Riemannian.matrix}
The coordinate representation $G(\xi) = \left( g_{ij}(\xi) \right)$ of the Riemannian metric $g$ is given by
\begin{equation} \label{eqn:Riemannian.matrix}
G(\xi) = \frac{-1}{\Pi^{(\alpha)}(\xi, \eta)} \left(I - \frac{\alpha}{\Pi^{(\alpha)}(\xi, \eta)} \eta \xi^{\top}\right) \frac{\partial \eta}{\partial \xi}.
\end{equation}
Its inverse is given by
\begin{equation} \label{eqn:Riemannian.matrix.inverse}
G^{-1}(\xi) = \left( g^{ij}(\xi) \right) = -\Pi^{(\alpha)} (\xi, \eta) \frac{\partial \xi}{\partial \eta} \left(I + \alpha \eta \xi^{\top}\right).
\end{equation}
\end{proposition}
\begin{proof}
Using the self-dual representation \eqref{eqn:divergence.self.dual}, we have
\begin{equation} \label{eqn:D.first.two.derivatives}
\begin{split}
\frac{\partial}{\partial \xi^i} \mathbf{D} \left[ \xi : \xi' \right] &= \frac{ \eta'^{i}}{\Pi(\xi, \eta')} - \frac{\partial \varphi}{\partial \xi^i}(\xi), \\
\frac{\partial^2}{\partial \xi^i \partial \xi'^j} \mathbf{D} \left[\xi : \xi'\right] &= \frac{1}{\Pi(\xi, \eta')^2} \left( \Pi(\xi, \eta') \frac{\partial \eta'^i}{\partial \xi'^j} - \alpha \eta'^i \sum_{\ell = 1}^d \xi^{\ell} \frac{\partial \eta'^{\ell}}{\partial \xi'^j}\right).
\end{split}
\end{equation}
Setting $\xi = \xi'$ and expressing in matrix form, we obtain \eqref{eqn:Riemannian.matrix}.

We obtain \eqref{eqn:Riemannian.matrix.inverse} by taking the inverse of \eqref{eqn:Riemannian.matrix}. For the middle term, we may use the Sherman-Morrison formula (see \eqref{eqn:Sherman.Morrison} below) to get
\[
\left(I - \frac{\alpha}{\Pi(\xi, \eta)} \eta \xi^{\top} \right)^{-1} = I + \alpha \eta \xi^{\top}.
\]
\end{proof}

\begin{remark} \label{rem:metric}
In Proposition \ref{prop:Riemannian.matrix} the Riemannian matrix is given in terms of the Jacobian matrix $\frac{\partial \eta}{\partial \xi}$. This is to emphasize the role of duality (compare \eqref{eqn:Riemannian.matrix} with \cite[(1.66)]{A16}) and to enable explicit expressions of the Christoffel symbols ${\Gamma_{ij}}^k = \Gamma_{ijm} g^{mk}$ which appear in the primal geodesic equations.

Let $\xi \in \Omega$ and $v \in \mathbb{R}^d$. By direct differentiation of \eqref{eqn:L.alpha.divergence}, we have
\begin{equation} \label{eqn:Riemannian.matrix.Euclidean}
\left.\frac{d^2}{dt^2} \mathbf{D}\left[ \xi + tv : \xi \right] \right|_{t = 0} = \frac{1}{2} v^{\top} \left( -D^2 \varphi(\xi) - \alpha (D\varphi(\xi))(D\varphi(\xi))^{\top}  \right) v.
\end{equation}
Thus we also have
\begin{equation} \label{eqn:concave.alpha.metric}
g_{ij}(\xi) = -\frac{\partial^2 \varphi}{\partial \xi^i \partial \xi^j}(\xi) - \alpha \frac{\partial \varphi}{\partial \xi^i }(\xi) \frac{\partial \varphi}{\partial \xi^j}(\xi)
\end{equation}
which is the positive definite matrix in \eqref{eqn:metric.exp.concave}. Using the Sherman-Morrison formula
\begin{equation} \label{eqn:Sherman.Morrison}
\left(A + uv^{\top}\right)^{-1} = A^{-1} - \frac{A^{-1} uv^{\top} A^{-1}}{1 + v^{\top} A^{-1} u },
\end{equation}
we can invert $G(\xi)$ to get
\[
G^{-1}(\xi) = - (D^2 \varphi)^{-1} + \alpha \frac{(D^2 \varphi)^{-1} (D\varphi) (D\varphi)^{\top} (D^2 \varphi)^{-1}}{1 + \alpha (D \varphi)^{\top} (D^2 \varphi)^{-1} (D \varphi)}.
\]
While this formula is explicit, it is difficult to use in differential-geometric computations. The same remark applies to the connections.
\end{remark}

\subsubsection{Affine connections}

\begin{proposition}
The Christoffel symbols of the primal connection are given by
\begin{equation} \label{eqn:primal.connection}
\Gamma_{ijk}(\xi) = \frac{\alpha}{\Pi^{(\alpha)}(\xi, \eta)^2} \eta^j \frac{\partial \eta^i}{\partial \xi^k} + \frac{\alpha}{\Pi^{(\alpha)}(\xi, \eta)^2} \eta^i \frac{\partial \eta^j}{\partial \xi^k} - \frac{2 \alpha^2}{\Pi^{(\alpha)}(\xi, \eta)^3} \eta^i \eta^j \sum_{\ell = 1}^d \xi^{\ell} \frac{\partial \eta^{\ell}}{\partial \xi^k},
\end{equation}
\begin{equation} \label{eqn:primal.connection2}
\begin{split}
{\Gamma_{ij}}^k(\xi) &= \frac{-\alpha}{\Pi^{(\alpha)}(\xi, \eta)} \left(\eta^i \delta_j^k + \eta^j \delta_i^k \right) = - \alpha \left(  \frac{\partial \varphi}{\partial \xi^i} \delta_j^k + \frac{\partial \varphi}{\partial \xi^j} \delta_i^k \right).
\end{split}
\end{equation}
where $\delta_{\cdot}^{\cdot}$ is the Kronecker delta.
\end{proposition}
\begin{proof}
From our previous computation \eqref{eqn:D.first.two.derivatives}, we have
\[
\frac{\partial}{\partial \xi^i} \frac{\partial}{\partial \xi'^k} \mathbf{D}\left[ \xi : \xi' \right] = \frac{1}{\Pi(\xi, \eta')} \frac{\partial \eta'^i}{\partial \xi'^k} - \frac{\alpha}{\Pi(\xi, \eta')^2} \eta'^i \sum_{\ell = 1}^d \xi^{\ell} \frac{\partial \eta'^{\ell}}{\partial \xi'^k}.
\]
Differentiating one more time and writing $\Pi = \Pi(\xi, \eta')$, we get
\begin{equation*}
\begin{split}
 \frac{\partial^2}{\partial \xi^i \partial \xi^j} \frac{\partial}{\partial \xi'^k} \mathbf{D}\left[ \xi : \xi' \right] = \frac{-\alpha}{\Pi^2} \eta'^j \frac{\partial \eta'^i}{\partial \eta'^k} - \frac{\alpha}{\Pi^2} \eta'^i \frac{\partial \eta'^j}{\partial \xi'^k} + \frac{2 \alpha^2}{\Pi^3} \eta'^i \eta'^j \sum_{\ell = 1}^d \xi^{\ell} \frac{\partial \eta'^{\ell}}{\partial \xi'^k}.
\end{split}
\end{equation*}
Setting $\xi = \xi'$, we obtain \eqref{eqn:primal.connection}.

To prove \eqref{eqn:primal.connection2}, we first use \eqref{eqn:Riemannian.matrix.inverse} to write
\[
g^{mk}(\xi) = -\Pi(\xi, \eta) \left( \frac{\partial \xi^m}{\partial \eta^k} + \alpha \sum_{\ell = 1}^d \frac{\partial \xi^m}{\partial \eta^{\ell}} \eta^{\ell} \xi^k \right).
\]
It follows that
\begin{equation*}
\begin{split}
& {\Gamma_{ij}}^k(\xi) = \Gamma_{ijm}(\xi) g^{mk}(\xi) \\
&= -\left( \frac{\alpha}{\Pi} \eta^j \frac{\partial \eta^i}{\partial \xi^m} + \frac{ \alpha}{\Pi} \eta^i \frac{\partial \eta^j}{\partial \xi^m} - \frac{2\alpha^2}{\Pi^2} \eta^i \eta^j \sum_{\ell} \xi^{\ell} \frac{\partial \eta^{\ell}}{\partial \xi^m} \right) \left( \frac{\partial \xi^m}{\partial \eta^k} + \alpha \sum_{\ell = 1}^d \frac{\partial \xi^m}{\partial \eta^{\ell}} \eta^{\ell} \xi^k \right) \\
&= \frac{-\alpha}{\Pi} \eta^j \delta_i^k - \frac{\alpha}{\Pi} \eta^i \delta_j^k + \frac{2\alpha^2}{\Pi^2} \eta^i \eta^j \xi^k - \frac{2\alpha^2}{\Pi} \eta^i \eta^j \xi^k  + \frac{2 \alpha^3}{\Pi^2} \eta^i \eta^j \xi^k \sum_{\ell = 1}^d \xi^{\ell} \eta^{\ell}.
\end{split}
\end{equation*}
The last three terms cancel out and the resulting expression gives the first equality in \eqref{eqn:primal.connection}. The second equality follows from the definition of $\eta$ as the $\alpha$-gradient \eqref{eqn:alpha.gradient} as well as the identity \eqref{eqn:nice.identity}.
\end{proof}

\subsection{Dualistic structure of $L^{(-\alpha)}$-divergence} \label{sec:geometry.minus.alpha}
Here we state the corresponding results for the $L^{(-\alpha)}$-divergence.

\begin{proposition}
For the dualistic structure generated by $\mathbf{D}^{(-\alpha)}$ where $\varphi$ is $\alpha$-exponentially convex, we have
\begin{equation} \label{eqn:convex.alpha.metric}
g_{ij}(\xi) = \frac{\partial^2 \varphi}{\partial \xi^i \partial \xi^j}(\xi) + \alpha \frac{\partial \varphi}{\partial \xi^i }(\xi) \frac{\partial \varphi}{\partial \xi^j}(\xi),
\end{equation}
\begin{equation} \label{eqn:convex.alpha.connection}
{\Gamma_{ij}}^k(\xi) = -\alpha \left(  \frac{\partial \varphi}{\partial \xi^i}(\xi) \delta_j^k + \frac{\partial \varphi}{\partial \xi^j}(\xi) \delta_i^k\right).
\end{equation}
\end{proposition}

Note that \eqref{eqn:primal.connection2} and \eqref{eqn:convex.alpha.connection} have the same form (even though here $\varphi$ is $\alpha$-exponentially convex). This is because ${\Gamma_{ij}}^k = \Gamma_{ijm} g^{mk}$ and the changes of signs cancel out in the product.

\section{Geometric consequences} \label{sec:geometric.consequences}
With all the coefficients available, we derive in this section the geometric properties of the dualistic structure $(g, \nabla, \nabla^*)$. We state the results for any $\mathbf{D}^{(\pm \alpha)}$ with $\alpha > 0$, but the proofs will again be given for the case of $\mathbf{D}^{(\alpha)}$. In Section \ref{sec:alpha.divergence} we specialize to the unit simplex and explain the connections with the $\alpha$-divergence.

\subsection{Dual projective flatness}
The following definition is taken from \cite{M99}.

\begin{definition} [Projective flatness] \label{def:projective.flatness}
Let $\nabla$ and $\widetilde{\nabla}$ be torsion-free affine connections on a smooth manifold. They are projectively equivalent if there exists a differential $1$-form $\tau$ such that
\begin{equation} \label{eqn:projective.equivalent}
\nabla_X Y = \widetilde{\nabla}_X Y + \tau(X) Y + \tau(Y) X
\end{equation}
for any smooth vector fields $X$ and $Y$. We say that $\nabla$ is projectively flat if $\nabla$ is projectively equivalent to a flat connection (a connection whose Riemann-Christoffel curvature tensor vanishes).
\end{definition}

If we write $\tau = a_i(\xi) d \xi^i$, say, using the primal coordinate system, then \eqref{eqn:projective.equivalent} is equivalent to
\begin{equation} \label{eqn:projective.equivalent.condition}
{\Gamma_{ij}}^k(\xi) = \widetilde{\Gamma}_{ij}\mathstrut^{k}(\xi) + a_i(\xi) \delta_j^k + a_j(\xi) \delta_i^k.
\end{equation}

\begin{theorem} \label{prop:projective.flat}
For any $\mathbf{D}^{(\pm \alpha)}$, the primal connection $\nabla$ and the dual connection $\nabla^*$ are projectively flat. Thus we say that the induced dualistic structures are dually projectively flat.
\end{theorem}
\begin{proof}
We only consider the primal connection. Consider the $1$-form defined by $\tau(X) = -\alpha X \varphi$, where $X \varphi$ means the derivative of $\varphi$ in the direction of $X$. In primal coordinates, we have
\begin{equation} \label{eqn:1-form}
\tau = -\sum_{i = 1}^d \alpha \frac{\partial \varphi}{\partial \xi^i} d\xi^i.
\end{equation}
Let $\widetilde{\nabla}$ be the flat Euclidean connection with respect to the coordinate system $\xi$ such that its Christoffel symbols satisfy $\widetilde{\Gamma}_{ij}\mathstrut^{k}(\xi) \equiv 0$. From \eqref{eqn:projective.equivalent.condition}, we see that $\nabla$ is projectively equivalent to $\widetilde{\nabla}$ and thus $\nabla$ is projectively flat.
\end{proof}

Projective flatness can be related to the behaviors of the geodesics. Recall that a (smooth) curve $\gamma: [0, 1] \rightarrow \mathcal{M}$ is said to be a primal geodesic if $\nabla_{\dot{\gamma}(t)} \dot{\gamma}(t) \equiv 0$. Equivalently, its primal coordinate representation $\xi(t) = (\xi^1(t), \ldots, \xi^d(t))$ satisfies the primal geodesic equations
\begin{equation} \label{eqn:primal.geodesic.equation}
\ddot{\xi}^k(t) + \dot{\xi}^i(t) \dot{\xi}^j(t) {\Gamma_{ij}}^k(\xi(t)) = 0, \quad k = 1, \ldots, d,
\end{equation}
where the dots denote derivatives with respect to time. Similarly, the curve $\gamma$ is a dual geodesic if $\nabla_{\dot{\gamma}(t)}^* \dot{\gamma}(t) \equiv 0$.

The following result is well-known in differential geometry.

\begin{lemma}
If $\nabla$ and $\widetilde{\nabla}$ are projectively equivalent, then a $\nabla$-geodesic is a $\widetilde{\nabla}$-geodesic up to a time reparameterization, and vice versa.
\end{lemma}

\begin{corollary}
For the dualistic structure induced by $\mathbf{D}^{(\pm \alpha)}$, if $\gamma$ is a primal (dual) geodesic, then its trace in the primal (dual) coordinate system is a straight line.
\end{corollary}

Thus the primal geodesics are straight lines under the primal coordinate system but run at non-constant speeds. Let us consider the time reparameterization. Using \eqref{eqn:primal.connection2}, we note that the primal geodesic equation \eqref{eqn:primal.geodesic.equation} (under the primal coordinate system) is equivalent to the single vector equation
\begin{equation} \label{eqn:primal.geodesic.equation.2}
\ddot{\xi}(t) = 2\alpha \dot{\xi}(t) \frac{d}{dt} \varphi(\xi(t)).
\end{equation}
From this we immediately see  that the trace of the primal geodesic is a straight line under the coordinate $\xi$. 

\begin{proposition}
Let $\gamma : [0, 1] \rightarrow \mathcal{M}$ be a primal geodesic. Then, in primal coordinates, we have
\begin{equation} \label{eqn:primal.geodesic.straight.line}
\xi(t) = (1 - h(t)) \xi(0) + h(t) \xi(1),
\end{equation}
where $h: [0, 1] \rightarrow [0, 1]$ is given by
\begin{equation} \label{eqn:primal.time.reparameterization}
h(t) = \frac{\int_0^t e^{2 \alpha \varphi(\xi(s))} ds}{\int_0^1 e^{2 \alpha \varphi(\xi(s)) } ds}.
\end{equation}
\end{proposition}
\begin{proof}
We already know that a primal geodesic can be written in the form \eqref{eqn:primal.geodesic.straight.line} for some time change $h$. Plugging this into the geodesic equation \eqref{eqn:primal.geodesic.equation.2}, we have
\[
\ddot{h}(t) = 2 \alpha \dot{h}(t) \frac{d}{dt} \varphi(\xi(t)) \Rightarrow  \frac{d}{dt} \log \dot{h}(t) = 2\alpha \frac{d}{dt} \varphi(\xi(t)).
\]
Integrating and using the fact that $h(0) = 0$ and $h(1) = 1$, we obtain the desired formula \eqref{eqn:primal.time.reparameterization}. 
\end{proof}

A natural question is the relationship between the primal and dual coordinate vector fields. This is studied in the following proposition.

\begin{proposition} \label{prop:inner.product} {\ } 
\begin{enumerate}
\item[(i)] For $\mathbf{D}^{(\alpha)}$, the inner product between $\frac{\partial}{\partial \xi^i}$ and $\frac{\partial}{\partial \eta^j}$ is
\begin{equation} \label{eqn:inner.product}
\left\langle \frac{\partial}{\partial \xi^i}, \frac{\partial}{\partial \eta^j} \right\rangle = \frac{-1}{\Pi(\xi, \eta)} \delta_{ij} + \frac{\alpha}{\Pi(\xi, \eta)^2} \xi^j \eta^i.
\end{equation}
\item[(ii)] For $\mathbf{D}^{(-\alpha)}$, the inner product between $\frac{\partial}{\partial \xi^i}$ and $\frac{\partial}{\partial \eta^j}$ is
\begin{equation} \label{eqn:inner.product2}
\left\langle \frac{\partial}{\partial \xi^i}, \frac{\partial}{\partial \eta^j} \right\rangle = \frac{1}{\Pi(\xi, \eta)} \delta_{ij} - \frac{\alpha}{\Pi(\xi, \eta)^2} \xi^j \eta^i.
\end{equation}
\end{enumerate}
\end{proposition}
\begin{proof}
Consider (i). Using Proposition \ref{prop:Riemannian.matrix}, we compute
\begin{equation*}
\begin{split}
\left\langle \frac{\partial}{\partial \xi^i}, \frac{\partial}{\partial \eta^j} \right\rangle &= \left\langle \frac{\partial}{\partial \xi^i}, \sum_m \frac{\partial \xi^m}{\partial \eta^j} \frac{\partial}{\partial \xi^m} \right\rangle = \sum_m \frac{\partial \xi^m}{\partial \eta^j} \left\langle \frac{\partial}{\partial \xi^i}, \frac{\partial}{\partial \xi^m} \right\rangle \\
&= \sum_m \frac{\partial \xi^m}{\partial \eta^j} \left( \frac{-1}{\Pi(\xi, \eta)} \frac{\partial \eta^i}{\partial \xi^m} + \frac{\alpha}{\Pi(\xi, \eta)^2} \eta^i \sum_{\ell} \xi^{\ell} \frac{\partial \eta^{\ell}}{\partial \xi^m} \right).
\end{split}
\end{equation*}
We obtain \eqref{eqn:inner.product} under some simplification.
\end{proof}

The formulas \eqref{eqn:inner.product} and \eqref{eqn:inner.product2} are quite interesting. In the dually flat case (i.e., $\mathbf{D}^{(0\pm)}$, also see \cite[Theorem 6.6]{A16}), the two coordinate fields are orthonormal in the sense that $\left\langle \frac{\partial}{\partial \xi^i}, \frac{\partial}{\partial \eta^j} \right\rangle = \pm \delta_{ij}$. When $\alpha > 0$, the first term of the inner product is conformal to $\pm \delta_{ij}$, but there is an extra term related to duality. This formula is the key ingredient in the proof of the generalized Pythagorean theorem.

\subsection{Curvatures} 
Next we consider the primal Riemann-Christoffel curvature tensor $R$ defined by
\begin{equation} \label{eqn:curvature.tensor}
R(X, Y)Z = \nabla_X \nabla_Y Z - \nabla_Y \nabla_X Z - \nabla_{[X, Y]} Z,
\end{equation}
where $X, Y, Z$ are vector fields and $[X, Y] = XY - YX$ is the Lie bracket. Under the primal coordinate system, we use the notation
\[
R\left( \frac{\partial}{\partial \xi^i}, \frac{\partial}{\partial \xi^j}\right) \frac{\partial}{\partial \xi^k} = {R_{ijk}}^{\ell} \frac{\partial}{\partial \xi^{\ell}}.
\]
It can be shown (see for example \cite[Section 5.8]{A16}) that the coefficients are given by
\begin{equation} \label{eqn:RC.coefficients}
{R_{ijk}}^{\ell} = \partial_i {\Gamma_{jk}}^{\ell} - \partial_j {\Gamma_{ik}}^{\ell} + {\Gamma_{jk}}^m {\Gamma_{im}}^{\ell} - {\Gamma_{ik}}^m {\Gamma_{jm}}^{\ell}.
\end{equation}
The dual curvature tensor $R^*$ is defined analogously.

\begin{proposition} { \ }
\begin{enumerate}
\item[(i)] For $\mathbf{D}^{(\alpha)}$, the primal curvature tensor is given by
\begin{equation} \label{eqn:primal.curvature}
{R_{ijk}}^{\ell}(\xi) = \alpha \left( g_{ik}(\xi) \delta_j^{\ell} - g_{jk}(\xi) \delta_i^{\ell}\right).
\end{equation}
\item[(ii)] For $\mathbf{D}^{(-\alpha)}$, the primal curvature tensor is given by
\begin{equation} \label{eqn:primal.curvature2}
{R_{ijk}}^{\ell}(\xi) = -\alpha \left( g_{ik}(\xi) \delta_j^{\ell} - g_{jk}(\xi) \delta_i^{\ell}\right).
\end{equation}
\end{enumerate}
\end{proposition}
\begin{proof}
This is a straightforward but lengthy computation using \eqref{eqn:primal.connection2}, \eqref{eqn:convex.alpha.connection} and \eqref{eqn:RC.coefficients}. We omit the details.
\end{proof}

\begin{definition} \label{def:constant.curvature}
When $\dim \mathcal{M} \geq 2$, a torsion-free connection has constant sectional curvature $k \in \mathbb{R}$ with respect to the metric $g$ if its curvature tensor $R$ satisfies the identity
\begin{equation} \label{eqn:constant.sectional.curvature}
R(X, Y)Z = k \left( \langle Y, Z \rangle X - \langle X, Z \rangle Y \right)
\end{equation}
for all vector fields $X$, $Y$ and $Z$.
\end{definition}

Using coordinate vector fields, we see that $R$ has constant curvature $k$ if and only if in any coordinate system we have
\begin{equation} \label{eqn:constant.curvature.def}
{R_{ijk}}^{\ell} = k \left( g_{jk} \delta_i^{\ell} - g_{ik} \delta_j^{\ell} \right).
\end{equation}
Comparing this with \eqref{eqn:primal.curvature} and \eqref{eqn:primal.curvature2}, we obtain

\begin{theorem} \label{thm:constant.curvature}
If $d = \dim \mathcal{M} \geq 2$, for $\mathbf{D}^{(\pm \alpha)}$ the primal and dual connections have constant sectional curvature $\mp \alpha$ with respect to the induced metric.
\end{theorem}

It is known that for a general dualistic structure $(\mathcal{M}, g, \nabla, \nabla^*)$, $\nabla$ has constant sectional curvature $k$ with respect to $g$ if and only if the same statement holds for $\nabla^*$. For a proof see \cite[Proposition 8.1.4]{CU14}. Some general properties of dualistic structures with constant curvatures are given in \cite[Theorem 9.7.2]{CU14}. For example, it is known that such a manifold is conjugate symmetric, conjugate Ricci-symmetric, and the connections $\nabla$ and $\nabla^*$ are equiaffine. In Section \ref{sec:characterize.geometry}, we show that if we assume in addition that the manifold is dually projectively flat in the sense of Theorem \ref{prop:projective.flat}, then the geometry can be characterized elegantly in terms of the $L^{(\pm \alpha)}$-divergence.

\subsection{Generalized Pythagorean theorem}
The generalized Pythagorean theorem is a fundamental result of information geometry and has numerous applications in information theory, statistics and machine learning. It was first proved for the Bregman divergence ($\mathbf{D}^{(\pm \alpha)}$ in our context) which induces a dually flat geometry. For an exposition of this beautiful result see \cite[Chapter 1]{A16}. In \cite{PW14, PW16} we introduced the $L$-divergence ($L^{(\alpha)}$-divergence when $\alpha = 1$) on the unit simplex and proved the generalized Pythagorean theorem. In \cite{PW16} the proof is quite different from that of the Bregman case and is rather involved because there we expressed the dualistic structure in terms of the exponential coordinate system. Here we give a unified and simplified treatment covering all $L^{(\pm \alpha)}$-divergences.

\begin{theorem} \label{thm:pyth}
Consider any $L^{(\pm \alpha)}$-divergence $\mathbf{D} = \mathbf{D}^{(\pm \alpha)}$ and the induced dualistic structure. Let $p, q, r \in \Omega$ and suppose that the dual geodesic from $q$ to $p$ exists  (this amounts to say that the line segment between $\eta_q$ and $\eta_p$ lies in the dual domain $\Omega'$. We need this assumption because $\Omega'$ may not be convex). Then the generalized Pythagorean relation
\begin{equation} \label{eqn:pyth.relation}
\mathbf{D}\left[q : p \right]  + \mathbf{D}\left[r : q \right]  = \mathbf{D}\left[r : p \right]
\end{equation}
holds if and only if the primal geodesic from $q$ to $r$ and the dual geodesic from $q$ to $p$ meet $g$-orthogonally at $q$.
\end{theorem}
\begin{proof}
Again we consider the case of $\mathbf{D}^{(\alpha)}$; the proof for $\mathbf{D}^{(-\alpha)}$ is similar. By the self-dual representation (Theorem \ref{thm:alpha.duality}), we have
\begin{equation*}
\begin{split}
\mathbf{D}\left[ q : p \right] &= \frac{1}{\alpha} \log \left(1 + \alpha \xi_q \cdot \eta_p \right) - \varphi(\xi_q) - \psi(\eta_p)
\end{split}
\end{equation*}
and similarly for $\mathbf{D}\left[r : q \right]$ and $\mathbf{D} \left[ r : p \right]$. Using these expressions and the identity \eqref{eqn:alpha.conjugate}, we see that the Pythagorean relation \eqref{eqn:pyth.relation} holds if and only if
\[
\left(1 + \alpha \xi_q \cdot \eta_p\right) \left(1 + \alpha \xi_r \cdot \eta_q\right) = \left(1 + \alpha \xi_r \cdot \eta_p\right)\left(1 + \alpha \xi_q \cdot \eta_q\right).
\]
Expanding and simplifying, we have
\begin{equation} \label{eqn:alpha.pyth.proof}
(\xi_r - \xi_q) \cdot (\eta_p - \eta_q) = \alpha ( \xi_q \cdot \eta_p) (\xi_r \cdot \eta_q) - \alpha (\xi_r \cdot \eta_p) (\xi_q \cdot \eta_q).
\end{equation}

On the other hand, consider the primal geodesic from $q$ to $r$. By projective flatness, in primal coordinates, the initial velocity vector (given by the inverse exponential map and expressed using the primal coordinate system) is proportional to $\xi_r - \xi_q$. Similarly, in dual coordinates, the initial velocity of the dual geodesic from $q$ to $p$ is proportional to $\eta_p - \eta_q$. By Proposition \ref{prop:inner.product}, we see that the two geodesics are orthogonal at $q$ if and only if
\begin{equation*}
\begin{split}
0 &= \left\langle (\xi^i_r - \xi^i_q) \frac{\partial}{\partial \xi^i}, (\eta_p^j - \eta_q^j) \frac{\partial}{\partial \eta^j} \right\rangle \\
  &= \sum_{i, j} (\xi_r^i - \xi_q^i) (\eta_p^j - \eta_q^j) \left( \frac{-1}{\Pi} \delta_{ij} + \frac{\alpha}{\Pi^2} \eta_q^i \xi_q^j \right) \\
  &= \frac{-1}{\Pi} (\xi_r - \xi_q) \cdot (\eta_p - \eta_q) + \frac{\alpha}{\Pi^2} \left( \sum_i (\xi_r^i - \xi_q^i) \eta_q^i \right) \left( \sum_j (\eta_p^j - \eta_q^j) \xi_q^j \right),
\end{split}
\end{equation*}
where $\Pi = \Pi(\xi_q, \eta_q) = 1 + \alpha \xi_q \cdot \eta_q$. Rearranging, we have
\begin{equation*}
\begin{split}
& (1 + \alpha \xi_q \cdot \eta_q) (\xi_r \cdot \eta_p - \xi_r \cdot \eta_q - \xi_q \cdot \eta_p + \xi_q \cdot \eta_q)\\
&= \alpha (\xi_r \cdot \eta_q - \xi_q \cdot \eta_q) (\xi_q \cdot \eta_p - \xi_q \cdot \eta_q).
\end{split}
\end{equation*}
The proof is completed by verifying that this and \eqref{eqn:alpha.pyth.proof} are equivalent.
\end{proof}

\subsection{The $\alpha$-divergence} \label{sec:alpha.divergence}
In this subsection we specialize to the unit simplex
\[
\mathcal{S}_d = \{p = (p^0, p^1, \ldots, p^d) \in (0, 1)^{1 + d} : p^0 + \cdots + p^d = 1\}.
\]
The following definition is taken from \cite[(3.39)]{A16}. Also see \cite[Definition 2.9]{AJLS17}.

\begin{definition} [$\alpha$-divergence]
Let $\alpha \neq \pm 1$. The $\alpha$-divergence is defined by
\begin{equation} \label{eqn:alpha.divergence}
\mathbf{D}_{\alpha} \left[ p : q \right] = \frac{4}{1 - \alpha^2} \left(1 - \sum_{i= 0}^d (p^i)^{\frac{1 - \alpha}{2}} (q^i)^{\frac{1 + \alpha}{2}}\right), \quad p, q \in \mathcal{S}_d.
\end{equation}
\end{definition}

We also consider the discrete R\'{e}nyi divergence given by
\begin{equation} \label{eqn:discrete.Renyi}
\mathbf{D}_{\tilde{\alpha}}\left( p || q\right) = \frac{1}{\tilde{\alpha} - 1} \log \left( \sum_{i = 0}^d p_i^{\tilde{\alpha}} q_i^{1 - \tilde{\alpha}} \right).
\end{equation}
Note the difference in the notations to avoid confusion. The following elementary lemma (whose proof is omitted) shows that the two divergences are monotone transformations of each other.

\begin{lemma} \label{lem:alpha.Renyi.identity}
For $\tilde{\alpha} \in (0, 1) \cup (1, \infty)$, let $\alpha = 1 - 2\tilde{\alpha} \in (-\infty, -1) \cup (-1, 1)$. Then for $p, q \in \mathcal{S}_d$ we have
\begin{equation} \label{eqn:Renyi.alpha.identity}
\mathbf{D}_{\tilde{\alpha}}\left(p || q\right) = \frac{1}{\tilde{\alpha} - 1} \log \left( 1 + \tilde{\alpha} (\tilde{\alpha} - 1) \mathbf{D}_{\alpha} \left[ p : q \right]\right).
\end{equation}
\end{lemma}

As in Example \ref{eg:simplex}, we may express $\mathbf{D}_{\tilde{\alpha}}\left(p || q\right)$ as the $L^{(\tilde{\alpha} - 1)}$-divergence of the $\mathcal{F}^{(\tilde{\alpha} - 1)}$-potential function when $\tilde{\alpha} \in (1, \infty)$, and as the  $L^{(-(1 - \tilde{\alpha}))}$-divergence of the $\mathcal{F}^{-(1 - \tilde{\alpha})}$-potential function when $ \tilde{\alpha} \in (0, 1)$.

Let $\alpha$ and $\tilde{\alpha}$ be as in Lemma \ref{lem:alpha.Renyi.identity}. Let $(g, \nabla, \nabla^*)$ be the dualistic structure induced on $\mathcal{S}_d$ by the $\alpha$-divergence $\mathbf{D}_{\alpha} \left[ p : q \right]$, and let $(\tilde{g}, \tilde{\nabla}, \tilde{\nabla}^*)$ be that induced by the R\'{e}nyi divergence $\mathbf{D}_{\tilde{\alpha}}\left(p || q\right)$. Since the two divergences are related by a monotone transformation, from \eqref{eqn:geometry.coefficients}, \eqref{eqn:Renyi.alpha.identity} and the chain rule we have $\tilde{g}_{ij} = \tilde{\alpha} g_{ij}$, $\tilde{\Gamma}_{ijk} = \tilde{\alpha} \Gamma_{ijk}$ and $\tilde{\Gamma}_{ijk}^* = \tilde{\alpha} \Gamma_{ijk}^*$. (Here the multiplier $\tilde{\alpha}$ is the derivative of the transformation \eqref{eqn:Renyi.alpha.identity} at $0$.) This implies that ${\Gamma_{ij}}^k = {\tilde{\Gamma}_{ij}}\mathstrut^{k}$ and ${\Gamma_{ij}^*}\mathstrut^{k} = {\tilde{\Gamma}_{ij}^*}\mathstrut^{k}$, so the two dualistic structures have the same primal and dual geodesics and the same primal and dual curvature tensors (see \eqref{eqn:RC.coefficients}). 

\begin{theorem}
For $\alpha = 1 - 2\tilde{\alpha} \in (-\infty, -1) \cup (-1, 1)$, the statistical manifold $(\mathcal{S}_d, g, \nabla, \nabla^*)$ induced by the $\alpha$-divergence \eqref{eqn:alpha.divergence} is dually projectively flat with constant curvature $\frac{1 - \alpha^2}{4}$.
\end{theorem}
\begin{proof}
We only indicate how the curvature is determined, as the rest follows immediately from other results of this section. We have $\tilde{\alpha} = \frac{1 - \alpha}{2}$. Suppose $\alpha \in (-\infty, -1)$. The R\'{e}nyi divergence $\mathbf{D}_{\tilde{\alpha}}\left(p || q\right)$ is then an $L^{(\alpha')}$-divergence where $\tilde{\alpha} = 1 + \alpha' \in (1, \infty)$. By Theorem \ref{thm:constant.curvature}, we have
\[
\tilde{R}(X, Y)Z = -\alpha' \left( \tilde{g}(Y, Z) - \tilde{g}(X, Z) Y \right).
\]
But $R = \tilde{R}$ and $\tilde{g} = \tilde{\alpha} g$. So we have
\[
R(X, Y)Z = -\alpha' \tilde{\alpha} \left( g(Y, Z) - g(X, Z) \right) = \frac{1 - \alpha^2}{4} \left( g(Y, Z) - g(X, Z) \right).
\]
This shows that the curvature is $\frac{1 - \alpha^2}{4}$. The case $\alpha \in (-1, 1)$ is similar.
\end{proof}

In this context, our geometry reduces to the geometry of the $\alpha$-divergence, and Theorem \ref{thm:pyth} implies the corresponding Pythagorean theorem. See in particular \cite[p.179]{AN02} where the curvature $\frac{1 - \alpha^2}{4}$ is determined using affine differential geometry. Their (8.67) is the product form of our \eqref{eqn:pyth.relation} for the R\'{e}nyi divergence \eqref{eqn:Renyi.alpha.identity}. Thus, our results provide a new approach to study the geometry (compare with \cite[Section 3.6]{AN02}). We plan to carry out a deeper study in future papers.

\section{Dually projectively flat manifolds with constant curvatures} \label{sec:characterize.geometry}
Using the expressions given in \eqref{eqn:geometry.coefficients}, any divergence $\mathbf{D}\left[ \cdot : \cdot \right]$ on a manifold $\mathcal{M}$ defines a dualistic structure $(\mathcal{M}, g, \nabla, \nabla^*)$. This operation is not one-to-one; it is easy to construct examples where two different divergences induce the same dualistic structure on $\mathcal{M}$. An interesting question is whether we can single out a divergence which is in some sense most natural  for a given dualistic structure. If so, we call it a canonical divergence. For dually flat manifolds, it is possible to show that there is a canonical divergence of Bregman type (see \cite[Section 6.6]{A16}). Thus, a dualistic structure is dually flat if and only if it is (locally) induced by a Bregman divergence.

In this section we ask and solve the same question where the manifold is dually projectively flat with constant curvature. We assume that the manifold $\mathcal{M}$ has dimension $d \geq 2$, and both $\nabla$ and $\nabla^*$ are projectively flat and have constant sectional curvature $\pm \alpha$ where $\alpha > 0$. (Note that if $\alpha = 0$ the manifold is dually flat and we reduce to the classical setting.) Our aim is to characterize the geometry based on these properties. The key idea is to use these properties to extend in a novel way the proofs of \cite[Theorem 6.2]{A16} and \cite[Theorem 4.3]{AJLS17} which address the dually flat case. Using tools of affine differential geometry, an alternative characterization is given in \cite[Corollary 1]{K90}.

\subsection{Dual coordinates and potential functions}

\begin{theorem} \label{thm:projectively.flat.characterize}
Consider a statistical manifold $(\mathcal{M}, g, \nabla, \nabla^*)$ such that both $\nabla$ and $\nabla^*$ are projectively flat and have constant negative sectional curvature $-\alpha$ where $\alpha > 0$. Then, near each point of $\mathcal{M}$, there exist local coordinate systems $\xi$ and $\eta$ such that the following statements hold:
\begin{enumerate}
\item[(i)] The primal geodesics are straight lines in the $\xi$ coordinates up to time reparameterizations, and the dual geodesics are straight lines in the $\eta$ coordinates up to time reparameterizations.
\item[(ii)] The coordinates satisfy the constraint $1 + \alpha \sum_{\ell = 1}^d \xi^{\ell} \eta^{\ell} > 0$.
\item[(iii)] There exist local $\alpha$-exponentially concave functions $\varphi(\xi)$ and $\psi(\eta)$ such that the generalized Fenchel identity
\begin{equation} \label{eqn:projectively.flat.fenchel}
\varphi(\xi) + \psi(\eta) \equiv c^{(\alpha)}(\xi, \eta) = \frac{1}{\alpha} \log \left(1 + \alpha \xi \cdot \eta \right)
\end{equation}
holds true. 
\item[(iv)] Under the coordinate systems $\xi$ and $\eta$ respectively, the Riemannian metric is given by
\begin{equation} \label{eqn:projectively.flat.metric}
\begin{split}
g_{ij}(\xi) &= -\frac{\partial^2 \varphi}{\partial \xi^i \partial \xi^j} (\xi) - \alpha \frac{\partial \varphi}{\partial \xi^i}(\xi) \frac{\partial \varphi}{\partial \xi^j}(\xi), \\
g_{ij}(\eta) &=  -\frac{\partial^2 \psi}{\partial \eta^i \partial \eta^j} (\eta) - \alpha \frac{\partial \psi}{\partial \eta^i}(\eta) \frac{\partial \psi}{\partial \eta^j}(\eta).
\end{split}
\end{equation}
\end{enumerate}
If the constant curvature is $+\alpha > 0$, analogous statements hold where the functions $\varphi$ and $\psi$ in (iii) are $\alpha$-exponentially convex, and in \eqref{eqn:projectively.flat.metric} the negative signs are replaced by positive signs.
\end{theorem}
\begin{proof}
Let the curvature be $-\alpha < 0$ (the proof for positive curvature is similar). Consider the primal connection $\nabla$ which is projectively flat. By Definition \ref{def:projective.flatness}, there exists a flat connection $\widetilde{\nabla}$ and a $1$-form $\tau$ such that \eqref{eqn:projective.equivalent} holds. As $\widetilde{\nabla}$ is flat, near each point there exists an affine coordinate system $\xi = (\xi^1, \ldots, \xi^d)$ under which the $\widetilde{\nabla}$-geodesics are constant velocity straight lines. By projective equivalence, the primal geodesics are straight lines under the $\xi$ coordinates up to time reparameterizations.  Furthermore, if we write $\tau = a_i(\xi) d\xi^i$ where the $a_i(\xi)$'s are smooth functions of $\xi$, by \eqref{eqn:projective.equivalent.condition} we have
\begin{equation} \label{eqn:Christoffel.projectively.flat}
{\Gamma_{ij}}^k(\xi) = a_i(\xi) \delta_j^k + a_j(\xi) \delta_i^k.
\end{equation}

Motivated by our previous result \eqref{eqn:1-form}, we want to show that
\begin{equation} \label{eqn:differential.form.claim}
a_i(\xi) = -\alpha \frac{\partial \varphi}{\partial \xi^i}(\xi)
\end{equation}
for some $\alpha$-exponentially concave function $\varphi$. To this end we will use the assumption that $\nabla$ has constant sectional curvature $-\alpha$ with respect to $g$.

Using the representation \eqref{eqn:Christoffel.projectively.flat} of the primal Christoffel symbols, we compute
\begin{equation} \label{eqn:R.computation}
\begin{split}
{R_{ijk}}^{\ell}(\xi) &= \partial_i {\Gamma_{jk}}^{\ell} - \partial_j {\Gamma_{ik}}^{\ell} + {\Gamma_{jk}}^m {\Gamma_{im}}^{\ell} - {\Gamma_{ik}}^m {\Gamma_{jm}}^{\ell} \\
&= (\partial_i a_j - \partial_j a_i) \delta_k^{\ell} + (\partial_i a_k - a_k a_i ) \delta_j^{\ell} - (\partial_j a_k - a_k a_j) \delta_i^{\ell}.
\end{split}
\end{equation}
On the other hand, since $\nabla$ has constant curvature $-\alpha$, by \eqref{eqn:constant.curvature.def} we have
\begin{equation} \label{eqn:R.form}
{R_{ijk}}^{\ell}(\xi) = \alpha \left( g_{ik}(\xi) \delta_j^{\ell} - g_{jk}(\xi) \delta_i^{\ell}\right).
\end{equation}
Equating \eqref{eqn:R.computation} and \eqref{eqn:R.form} and using the assumption that $\dim M = d \geq 2$ (so that we can pick different indices), we see that
\begin{equation} \label{eqn:closed.form}
\frac{\partial a_j}{\partial \xi^i} (\xi) = \frac{\partial a_i}{\partial \xi^j} (\xi)
\end{equation}
and
\begin{equation} \label{eqn:g.expression}
\alpha g_{ij}(\xi) = \partial_i a_j(\xi) - a_i(\xi) a_j(\xi).
\end{equation}

By \eqref{eqn:closed.form}, the $1$-form $\tau = a_i(\xi) d\xi^i$ is closed. So, locally, it is exact. Thus, there exists a locally defined function $\varphi(\xi)$ such that our claim \eqref{eqn:differential.form.claim} holds. In particular, the $1$-form $\tau$ is locally given by $\tau(X) = -\alpha X \varphi$.

Now we may write the metric \eqref{eqn:g.expression} in the form
\begin{equation} \label{eqn:g.expression.new}
g_{ij}(\xi) = -\frac{\partial^2 \varphi}{\partial \xi^i \partial \xi^j} (\xi) - \alpha \frac{\partial \varphi}{\partial \xi^i}(\xi) \frac{\partial \varphi}{\partial \xi^j}(\xi).
\end{equation}
Since the matrix $\left( g_{ij}(\xi) \right)$ is strictly positive definite, from \eqref{eqn:Phi.D2} we have that $D^2 e^{\alpha \varphi} < 0$. In particular, $\varphi$ is locally $\alpha$-exponentially concave.

Without loss of generality we may assume $0$ belongs to the domain of $\varphi$. This implies (see \eqref{eqn:L.divergence.motivation}) that
\[
1 - \alpha D\varphi(\xi) \cdot \xi > 0
\]
and the $\alpha$-gradient
\begin{equation} \label{eqn:projectively.flat.dual.variable}
\eta := \nabla^{(\alpha)} \varphi(\xi) = \frac{1}{1 - \alpha D\varphi(\xi) \cdot \xi} \nabla \varphi (\xi)
\end{equation}
is well-defined. As in \eqref{eqn:nice.identity} we have $ 1+ \alpha \xi \cdot \eta > 0$.

With the dual coordinate $\eta$ defined by \eqref{eqn:projectively.flat.dual.variable}, we define the dual function $\psi(\eta)$ by the generalized Fenchel identity:
\begin{equation} \label{eqn:dual.alpha.function}
\psi(\eta) = c^{(\alpha)}(\xi, \eta) - \varphi(\xi).
\end{equation}
By $\alpha$-duality $\psi$ is locally $\alpha$-exponentially concave, and the Fenchel identity \eqref{eqn:projectively.flat.fenchel} holds by construction.

Since the dual connection $\nabla^*$ is uniquely determined given $g$ and $\nabla$ (see \cite[(6.6)]{A16}), with the results in Section \ref{sec:L.geometry} it is a now routine exercise to verify that $\eta$ is an affine coordinate system for the projectively flat dual connection $\nabla^*$, and the metric is also given by the second formula in \eqref{eqn:projectively.flat.metric}. This completes the proof of the theorem.

\end{proof}

\subsection{Canonical divergence}
Using Theorem \ref{thm:projectively.flat.characterize}, we can define locally a canonical divergence which generalizes the one for a dually flat manifold. It is interesting to know whether this canonical divergence coincides with the one defined by Ay and Amari in \cite{AA15} for a generic dualistic structure.

\begin{theorem} [canonical divergence] \label{thm:canonical.divergence}
Consider the setting of Theorem \ref{thm:projectively.flat.characterize} where $\alpha > 0$. Let $\xi$, $\eta$, $\varphi$ and $\psi$ be given as in the theorem.
\begin{enumerate}
\item[(i)] If the curvature is $-\alpha < 0$, we define the local $L^{(\alpha)}$-divergence
\begin{equation} \label{eqn:alpha.canonical.divergence}
\mathbf{D} \left[ q : p \right] = \frac{1}{\alpha} \log \left(1 + \alpha \xi_q \cdot \eta_p\right) - \varphi(\xi_q) - \psi(\eta_p).
\end{equation}
\item[(ii)] If the curvature is $\alpha > 0$, we define the local $L^{(-\alpha)}$-divergence
\begin{equation} \label{eqn:alpha.canonical.divergence2}
\mathbf{D} \left[ q : p \right] = \varphi(\xi_q) + \psi(\eta_p) - \frac{1}{\alpha} \log \left(1 + \alpha \xi_q \cdot \eta_p\right).
\end{equation}
\end{enumerate}
Then these divergences are independent of the choices of $\xi$, $\eta$, $\varphi$ and $\psi$. Furthermore, the canonical divergence induces the given Riemannian metric and the primal and dual connections.
\end{theorem}
\begin{proof}
Note that the primal affine coordinate system $\xi$ is unique up to an affine transformation. The rest are then determined by \eqref{eqn:projectively.flat.dual.variable} and \eqref{eqn:dual.alpha.function}. By a direct computation, we can verify that \eqref{eqn:alpha.canonical.divergence} and \eqref{eqn:alpha.canonical.divergence2} remain invariant.
\end{proof}

\begin{acknowledgements}
Some of the results were obtained when the author was teaching a PhD topics course on optimal transport and information geometry at the University of Southern California. He thanks his students for their interest and feedbacks. He also thanks John Man-shun Ma, Jeremy Toulisse and Soumik Pal for helpful discussions. He is grateful to Shun-ichi Amari for his insightful comments and for pointing out the connection with the $\alpha$-divergence. Finally, he thanks the anonymous referees for their helpful comments which improved the paper.
\end{acknowledgements}

\bibliographystyle{spmpsci}      
\bibliography{infogeo.new}   

\end{document}